\def\isarxiv{1}

\ifdefined\isarxiv
\documentclass[11pt]{article}

\usepackage{amsthm, amsmath, amssymb, graphicx, url}
\usepackage[margin=1in]{geometry}

\else
\documentclass[final]{colt2024} 

\fi

\usepackage{latexsym, amscd, amsfonts, mathrsfs, stmaryrd, tikz-cd, mathrsfs, bbm, esint, listings, moreverb, hyperref, pifont, enumitem}

\def\bE{\mathbb{E}}
\def\bP{\mathbb{P}}
\def\bR{\mathbb{R}}
\def\bZ{\mathbb{Z}}

\def\cA{\mathcal{A}}

\def\cN{\mathcal{N}}

\def\cX{\mathcal{X}}
\def\cY{\mathcal{Y}}
\def\cZ{\mathcal{Z}}

\DeclareFontFamily{U}{mathx}{}
\DeclareFontShape{U}{mathx}{m}{n}{<-> mathx10}{}
\DeclareSymbolFont{mathx}{U}{mathx}{m}{n}
\DeclareMathAccent{\widehat}{0}{mathx}{"70}
\DeclareMathAccent{\widecheck}{0}{mathx}{"71}

\def\wh{\widehat}
\def\wt{\widetilde}

\newcommand{\circnum}[1]{%
  \text{\ding{\the\numexpr #1+191}}%
}

\DeclareMathOperator\BP{\mathrm{BP}}

\DeclareMathOperator\BOHT{\mathrm{BOHT}}

\DeclareMathOperator\BSC{\mathrm{BSC}}

\DeclareMathOperator\HSBM{\mathrm{HSBM}}
\DeclareMathOperator\Id{\mathrm{Id}}
\DeclareMathOperator\KL{\mathrm{KL}}

\DeclareMathOperator\Pois{\mathrm{Pois}}

\DeclareMathOperator\Unif{\mathrm{Unif}}

\ifdefined\isarxiv
\newtheorem{theorem}{Theorem}
\newtheorem{lemma}[theorem]{Lemma}
\newtheorem{proposition}[theorem]{Proposition}
\newtheorem{corollary}[theorem]{Corollary}

\theoremstyle{definition}

\newtheorem{condition}[theorem]{Condition}

\else

\fi

\usepackage{times}

\begin{document}

\ifdefined\isarxiv
\title{Community detection in the hypergraph stochastic block model and reconstruction on hypertrees}
\date{}
\author{
Yuzhou Gu\thanks{\texttt{yuzhougu@ias.edu}. Institute for Advanced Study.}
\and
Aaradhya Pandey\thanks{\texttt{ap9898@princeton.edu}. Princeton University.}
}
\else
\title[Community detection in the HSBM and reconstruction on hypertrees]{Community detection in the hypergraph stochastic block model and reconstruction on hypertrees}
\coltauthor{%
 \Name{Yuzhou Gu} \Email{yuzhougu@ias.edu}\\
 \addr Institute for Advanced Study
 \AND
 \Name{Aaradhya Pandey} \Email{ap9898@princeton.edu}\\
 \addr Princeton University%
}
\fi

\maketitle

\begin{abstract}
  We study the weak recovery problem on the $r$-uniform hypergraph stochastic block model ($r$-HSBM) with two balanced communities. In this model, $n$ vertices are randomly divided into two communities, and size-$r$ hyperedges are added randomly depending on whether all vertices in the hyperedge are in the same community. The goal of weak recovery is to recover a non-trivial fraction of the communities given the hypergraph.
  \cite{pal2021community,stephan2022sparse} established that weak recovery is always possible above a natural threshold called the Kesten-Stigum (KS) threshold.
  For assortative models (i.e., monochromatic hyperedges are preferred), \cite{gu2023weak} proved that the KS threshold is tight if $r\le 4$ or the expected degree $d$ is small.
  For other cases, the tightness of the KS threshold remained open.

  In this paper we determine the tightness of the KS threshold for a wide range of parameters.
  We prove that for $r\le 6$ and $d$ large enough, the KS threshold is tight. This shows that there is no information-computation gap in this regime and partially confirms a conjecture of \cite{angelini2015spectral}.
  On the other hand, we show that for $r\ge 5$, there exist parameters for which the KS threshold is not tight.
  In particular, for $r\ge 7$, the KS threshold is not tight if the model is disassortative (i.e., polychromatic hyperedges are preferred) or $d$ is large enough.
  This provides more evidence supporting the existence of an information-computation gap in these cases.


  Furthermore, we establish asymptotic bounds on the weak recovery threshold for fixed $r$ and large $d$. We also obtain a number of results regarding the broadcasting on hypertrees (BOHT) model, including the asymptotics of the reconstruction threshold for $r\ge 7$ and impossibility of robust reconstruction at criticality.
\end{abstract}

\ifdefined\isarxiv
\else
\fi


\section{Introduction} \label{sec:intro}
\subsection{Hypergraph stochastic block model} \label{sec:intro:hsbm}
The stochastic block model (SBM) is a random graph model with community structures. It is the canonical model for studying community detection, and has applications in statistics, machine learning, and network science. The theoretical study of community detection in the SBM has received a lot of attention in the past decade since \cite{decelle2011asymptotic}; see \cite{abbe2017community} for a survey.

The hypergraph stochastic block model (HSBM) is a generalization of the SBM to hypergraphs. In this paper, we study the $r$-uniform HSBM with two symmetric communities, defined as follows.
Let $n\in \bZ_{\ge 0}$ denote the number of vertices and $V=[n]$ be the set of vertices\footnote{We use the notation $[n]=\{1,\ldots,n\}$.}.
Let $r\in \bZ_{\ge 2}$ be the hyperedge size and $a,b\in \bR_{\ge 0}$ be two parameters\footnote{While several previous works (e.g., \cite{pal2021community,zhang2022sparse,gu2023weak}) assume that $a\ge b$, we do not make this assumption.}.
Generate a random label $X_u\in \{\pm\}$ for all vertices $u\in V$ i.i.d.~$\sim \Unif(\{\pm\})$.
For every subset $S\subseteq V$ of size $r$, the hyperedge $S$ is added with probability $\frac a{\binom n{r-1}}$ if all vertices in $S$ have the same label; otherwise add the hyperedge with probability $\frac b{\binom n{r-1}}$. The model outputs $(X,G)$, where $G$ is the resulting hypergraph. We denote the model as $\HSBM(n,r,a,b)$.

We define two useful derived parameters $d$ and $\lambda$.
\begin{align} \label{eqn:hsbm-params-d-lambda}
  d = \frac{a-b + 2^{r-1} b}{2^{r-1}}, \qquad \lambda = \frac{a-b}{a-b+2^{r-1} b}.
\end{align}
The expected degree of every vertex is $d \pm o(1)$, and $\lambda$ is a strength parameter with range $\left[-\frac 1{2^{r-1}-1}, 1\right]$.

For the constant degree regime we consider, the most relevant community detection question is the weak recovery problem, which asks to recover a non-trivial fraction of the community labels (up to a global sign flip) given the hypergraph $G$. Formally, the goal is to design a (potentially inefficient) estimator $\wh X = \wh X(G)$ such that for a constant $c < \frac 12$,
\begin{gather}
  \lim_{n\to \infty} \bP\left[d_H(\wh X, X) \le (c+o(1))n\right] = 1,\\
  \text{where}\qquad d_H(X,Y) = \min_{s\in \{\pm\}} \sum_{u\in V} \mathbbm{1}\{X_u\ne s Y_u\}.
\end{gather}

For $r=2$, the model $\HSBM(n,2, a,b)$ becomes SBM with two symmetric communities, where the weak recovery threshold was established by \cite{massoulie2014community,mossel2015reconstruction,mossel2018proof}: weak recovery is possible if and only if $d\lambda^2 > 1$.
The threshold $d\lambda^2=1$ is called the Kesten-Stigum (KS) threshold and originates from the study of multitype Galton-Watson processes (\cite{kesten1966additional}).
For the HSBM, the natural generalization of the KS threshold is $(r-1) d \lambda^2 = 1$.
\cite{angelini2015spectral} conjectured that a phase transition for weak recovery occurs at the KS threshold.
\cite{pal2021community,stephan2022sparse} proved the positive part of the conjecture, that is, when $(r-1) d \lambda^2 > 1$, weak recovery is possible and can be achieved using an efficient (polynomial time) algorithm.
On the negative part, \cite{gu2023weak} proved that for assortative models (i.e., $a>b$), when $r=3,4$ or $\lambda\ge \frac 15$, if $(r-1) d \lambda^2 \le 1$, then weak recovery is impossible.
They also provided evidence (via studying the reconstruction on hypertrees problem, see Section~\ref{sec:intro:boht}) suggesting that for $r\ge 7$ and $d$ large enough, weak recovery might be information-theoretically possible below the KS threshold.
Nevertheless, the tightness of the KS threshold for disassortative models (i.e., $a<b$) or assortative models with $r\ge 5$ and large $d$ remained open before the current work.

Our first contribution is to show that the KS threshold is tight for $r\le 6$ and $d$ large enough.
\begin{theorem}[KS is tight for $r\le 6$ and large degree, HSBM] \label{thm:hsbm-r56}
  Consider the model $\HSBM(n,r,a,b)$ where $r\le 6$ and $a,b\in \bR_{\ge 0}$.
  There exists $d_0 = d_0(r)$ such that for all $d\ge d_0$, if $(r-1) d \lambda^2 \le 1$, then weak recovery is impossible.
\end{theorem}

Our second contribution is to provide an upper bound for HSBM weak recovery.
For $x,y\in [0,1]$, define the binary Kullback-Leibler (KL) divergence function $d_{\KL}(\cdot \| \cdot)$ as\footnote{Throughout the paper, $\log$ denotes natural logarithm.}
\begin{align} \label{eqn:binary-kl-div-func}
  d_{\KL}(x \| y) = x \log \frac xy + (1-x) \log \frac {1-x}{1-y}.
\end{align}
\begin{theorem}[HSBM weak recovery upper bound] \label{thm:hsbm-upper}
  Consider the model $\HSBM(n,r,a,b)$ with $a,b\in \bR_{\ge 0}$.
  If
  \begin{align} \label{eqn:thm-hsbm-upper:cond}
    \frac dr \cdot d_{\KL}\left( \lambda + \frac{1-\lambda}{2^{r-1}} \left\| \frac 1{2^{r-1}} \right. \right) > \log 2,
  \end{align}
  then weak recovery is possible.
\end{theorem}
Comparing Eq.~\eqref{eqn:thm-hsbm-upper:cond} with the KS threshold $(r-1) d \lambda^2 = 1$, we obtain that the KS threshold is not tight for a wide range of parameters.
\begin{corollary}[Non-tightness of the KS threshold for $r\ge 5$, HSBM] \label{coro:hsbm-upper}
  Work in the setting of Theorem~\ref{thm:hsbm-upper}.
  If $r\ge 5$, then there exists $\lambda_0=\lambda_0(r) \in \left(-\frac 1{2^{r-1}-1}, 1\right]$ such that for any non-zero $\lambda\in \left[-\frac 1{2^{r-1}-1}, \lambda_0\right)$, there exists $d$ satisfying $(r-1) d \lambda^2<1$ such that weak recovery is possible.
  Furthermore, for $r\ge 7$ we can take $\lambda_0(r) > 0$.
\end{corollary}
In Corollary~\ref{coro:hsbm-upper}, we can take $\lambda_0(r)$ to be the minimum solution in $\left[-\frac 1{2^{r-1}-1}, 1\right]$ to the equation
\begin{align} \label{eqn:hsbm-upper-ks}
  d_{\KL}\left( \lambda + \frac{1-\lambda}{2^{r-1}} \left\| \frac 1{2^{r-1}} \right. \right) = r(r-1)\lambda^2 \log 2.
\end{align}
Table~\ref{tab:hsbm-upper} presents numerical values of $\lambda_0(r)$ for small $r$.

\begin{table}[ht]
  \centering
  \begin{tabular}{|c|c|c|c|c|c|c|} \hline
    $r$ & $5$ & $6$ & $7$ & $8$ & $9$ & $10$\\ \hline
    $\lambda_0$ & $-0.06575$ & $-0.02154$ & $0.00413$ & $0.01920$ & $0.02807$ & $0.03320$ \\ \hline
  \end{tabular}
  \caption{Minimum solution to Eq.~\eqref{eqn:hsbm-upper-ks} for $5\le r\le 10$, rounded down to the nearest multiple of $10^{-5}$. For non-zero $\lambda\in \left[-\frac 1{2^{r-1}-1}, \lambda_0\right)$, the KS threshold is not tight. See Corollary~\ref{coro:hsbm-upper} for the complete statement.}
  \label{tab:hsbm-upper}
\end{table}

Furthermore, we obtain the following asymptotic lower and upper bounds for the HSBM weak recovery threshold.
\begin{theorem}[Asymptotic HSBM weak recovery bounds] \label{thm:hsbm-asymp}
  Consider the model $\HSBM(n,r,a,b)$ with $r\ge 2$ and $a,b\in \bR_{\ge 0}$.
  \begin{enumerate}[label=(\roman*)]
    \item (Upper bound) For any $\epsilon>0$, there exists $d_0=d_0(r,\epsilon)$, such that for all $d\ge d_0$, if $d \lambda^2 \ge \frac{2r\log 2}{2^{r-1}-1} + \epsilon$, then weak recovery is possible. \label{item:thm-hsbm-asymp:upper}
    \item (Lower bound) For any $\epsilon>0$, there exists $d_0=d_0(r,\epsilon)$, such that for all $d\ge d_0$, if $d \lambda^2 \le \frac 1{2^{r-2}} - \epsilon$, then weak recovery is impossible. \label{item:thm-hsbm-asymp:lower}
  \end{enumerate}
\end{theorem}

Our current knowledge about the tightness of the KS threshold for HSBM weak recovery is summarized in Table~\ref{tab:tightness-ks}. Our contribution is the ``large $d$'' column (except for the $r=3,4$, $\lambda>0$ case which was established in \cite{gu2023weak}) and all non-tightness results.

\begin{table}[ht]
  \centering
  \begin{tabular}{|c|c|c|c|} \hline
    & any $d$ &  small $d$ & large $d$\\ \hline
    $r=3,4$, $\lambda>0$ & always tight & (tight) & (tight) \\
    $r=3,4$, $\lambda<0$ & unknown & unknown & tight \\
    $r=5,6$, $\lambda>0$ & unknown & tight & tight \\
    $r=5,6$, $\lambda<0$ & unknown & not tight & tight \\
    $r\ge 7$, $\lambda>0$ & unknown & tight & not tight \\
    $r\ge 7$, $\lambda<0$ & never tight & (not tight) & (not tight) \\ \hline
  \end{tabular}
  \caption{Tightness of the KS threshold for HSBM weak recovery. An entry in parentheses means it is implied by another entry in the same row. References: Possibility above KS: \cite{pal2021community} ($\lambda>0$), \cite{stephan2022sparse} (general HSBM). Tightness for $r=3,4$, $\lambda>0$, and $r\ge 5$, $\lambda>0$, small $d$: \cite{gu2023weak}. Tightness for $r\le 6$, large $d$: Theorem~\ref{thm:hsbm-r56}. Non-tightness results: Corollary~\ref{coro:hsbm-upper}. This table is also our current knowledge about tightness of the KS threshold for BOHT reconstruction.}
  \label{tab:tightness-ks}
\end{table}

Proofs of Theorem~\ref{thm:hsbm-r56} and Theorem~\ref{thm:hsbm-asymp}\ref{item:thm-hsbm-asymp:lower} are based on the reconstruction on hypertrees problem, which we will discuss in Section~\ref{sec:intro:boht}.
Proof of Theorem~\ref{thm:hsbm-upper} uses a simple estimator which returns a balanced partition with expected in-community hyperedge density. We show that when Eq.~\eqref{eqn:thm-hsbm-upper:cond} is satisfied, any balanced partition that achieves expected in-community hyperedge density is correlated with the true partition. This method was previously used in \cite{banks2016information} in the case of $q$-community SBM.
Corollary~\ref{coro:hsbm-upper} and Theorem~\ref{thm:hsbm-asymp}\ref{item:thm-hsbm-asymp:upper} follow from Theorem~\ref{thm:hsbm-upper}.

Another interesting community detection problem for the HSBM is the testing problem, which asks to distinguish an HSBM hypergraph from an Erd\H{o}s-R\'enyi hypergraph with the same hyperedge density.
\cite{yuan2022testing} studied the testing problem for a class of symmetric HSBMs (which includes our two-community case), proving that testing (with $o(1)$ error probability) is possible above the KS threshold. They showed that the KS threshold is tight for the two-community three-uniform HSBM, and that the KS threshold is not tight in general. In fact, they proved that when Eq.~\eqref{eqn:thm-hsbm-upper:cond} holds, testing is possible.\footnote{We remark that \cite[Theorem 2.6, Part 1, second half]{yuan2022testing} is off by a factor of two because of a computation error. In particular, in the second-to-last display in the proof of Theorem 2.6, the third step performs Taylor expansion to order $\epsilon$. However, because the final result is of order $\epsilon^2$, to get the correct coefficient in the final result, one needs to perform Taylor expansion to order $\epsilon^2$.} Their results do not imply our Theorem~\ref{thm:hsbm-upper} because a direct relationship between testing and weak recovery is not known. Nevertheless, this provides evidence suggesting that there might be a close relationship between the two problems.

\subsection{Reconstruction on hypertrees} \label{sec:intro:boht}
Theorem~\ref{thm:hsbm-r56} and Theorem~\ref{thm:hsbm-asymp}\ref{item:thm-hsbm-asymp:lower} are proved by studying the reconstruction problem for the broadcasting on hypertrees (BOHT) model, defined as follows.
Let $r\in \bZ_{\ge 2}$ be the hyperedge size, $\lambda\in \left[-\frac 1{2^{r-1}-1}, 1\right]$ be a strength parameter, $D$ be a distribution on $\bZ_{\ge 0}$, which is usually taken to be the point distribution at $d\in \bZ_{\ge 0}$ or the Poisson distribution $\Pois(d)$ with expectation $d\in \bR_{\ge 0}$.
Let $T$ be the $r$-uniform Galton-Watson linear\footnote{Linear means the intersection of two different hyperedges has size at most one.} hypertree, where every vertex independently has $b\sim D$ downward hyperedges. The BOHT model assigns a $\pm$ label to every vertex according to the following process.
The root $\rho$ has label $\sigma_\rho$ generated from the uniform distribution $\pi = \Unif(\{\pm\})$.
Let $B_{r,\lambda}: \{\pm\} \to \{\pm\}^{r-1}$ be the channel
\begin{align} \label{eqn:broadcast-channel}
  B_{r,\lambda}(y_1,\ldots,y_{r-1} | x) = \left\{
    \begin{array}{ll}
      \lambda + \frac{1-\lambda}{2^{r-1}}, & \text{if}~y_1=\cdots=y_{r-1}=x,\\
      \frac{1-\lambda}{2^{r-1}}, & \text{otherwise}.
    \end{array}
  \right.
\end{align}
Suppose we have generated label $\sigma_u$ of a vertex $u$. Then for every downward hyperedge $\{u,v_1,\ldots,v_{r-1}\}$, children labels $\sigma_{v_1},\ldots,\sigma_{v_{r-1}}$ are generated according to $B_{r,\lambda}$, that is, $(\sigma_{v_1},\ldots, \sigma_{v_{r-1}}) \sim B_{r,\lambda}(\cdot | \sigma_u)$.
In this way every vertex has received a label.
The model outputs $(T,\sigma)$ and we denote the model as $\BOHT(r,\lambda,D)$. When $D$ is a point distribution at $d\in \bZ_{\ge 0}$ we denote the model as $\BOHT(r,\lambda,d)$.

The reconstruction problem is an important problem for the BOHT model. Let $(T,\sigma) \sim \BOHT(r,\lambda,D)$.
For $k\in \bZ_{\ge 0}$, let $T_k$ denote the tree structure up to the $k$-th level (the root being at the $0$-th level), and $L_k$ denote the set of vertices at the $k$-th level. We say reconstruction is possible if
\begin{align}
  \lim_{k\to \infty} I(\sigma_\rho; T_k, \sigma_{L_k})
\end{align}
is positive, and impossible otherwise. Here $I(\cdot;\cdot)$ denotes mutual information. In other words, reconstruction is possible if and only if we can reconstruct the root label $\sigma_\rho$ with $\frac 12-\Omega(1)$ error probability given the leaf labels $\sigma_{L_k}$, as $k\to \infty$.

The reconstruction problem is useful for the study of HSBM weak recovery. Consider the model $\HSBM(n,r,a,b)$. Let $d,\lambda$ be defined as in Eq.~\eqref{eqn:hsbm-params-d-lambda}. Consider the model $\BOHT(r,\lambda,\Pois(d))$, which we call the BOHT model corresponding to $\HSBM(n,r,a,b)$.
If reconstruction is impossible for $\BOHT(r,\lambda,\Pois(d))$, then weak recovery is impossible for $\HSBM(n,r,a,b)$. This relationship was first established by \cite{mossel2015reconstruction,mossel2018proof} for the $r=2$ case, and their proof can be easily generalized. See \cite[Theorem 5.15]{gu2023channel} for a proof for general HSBMs.
We note that the converse is not true in general: for the $q$-community SBM, there exist parameters where weak recovery is impossible for the SBM but reconstruction is possible for the corresponding broadcasting on trees (BOT) model.


In the $r=2$ case, the model $\BOHT(2,\lambda, D)$ is called the Ising model, where the reconstruction threshold was established by \cite{bleher1995purity,evans2000broadcasting}: reconstruction is possible if and only if $d \lambda^2 > 1$. The threshold $d \lambda^2 = 1$ is also called the Kesten-Stigum (KS) threshold.
For $r\ge 3$, the natural generalization of the KS threshold is $(r-1) d \lambda^2 = 1$.
In general, it is known that reconstruction is possible whenever $(r-1) d \lambda^2 > 1$, and the tightness of the KS threshold is an interesting question.
\cite{gu2023weak} proved that reconstruction for $\BOHT(r,\lambda,D)$ is impossible for $(r-1) d \lambda^2 \le 1$ if $r\le 4$, $\lambda>0$, or $\lambda\ge \frac 15$.
Furthermore, they proved that the KS threshold is not tight for $r\ge 7$ and $d$ large enough.
The remaining cases were left open.

Our main contribution on the BOHT model is tightness of the KS threshold for $r\le 6$ and large $d$.
\begin{theorem}[KS is tight for $r\le 6$ and large degree, BOHT] \label{thm:boht-r56}
  Consider the model $\BOHT(r,\lambda,d)$ or $\BOHT(r,\lambda,\Pois(d))$ where $r\le 6$ and $\lambda\in \left[-\frac 1{2^{r-1}-1}, 1\right]$.
  There exists $d_0 = d_0(r)$ such that for all $d\ge d_0$, if $(r-1) d \lambda^2 \le 1$, then reconstruction is impossible.
\end{theorem}

Not coincidentally, Table~\ref{tab:tightness-ks} is also our current knowledge about tightness of the KS threshold for the reconstruction problem on $\BOHT(r,\lambda,D)$ (for $D=d$ or $D=\Pois(d)$). Possibility of reconstruction above KS is folklore by a counting estimator. Tightness for $r=3,4$, $\lambda>0$, and $r\ge 5$, $\lambda>0$, small $d$, and non-tightness for $r\ge 7$, large $d$ are by \cite{gu2023weak}. Tightness for $r\le 6$, large $d$ (except for $r=3,4$, $\lambda>0$) is by Theorem~\ref{thm:boht-r56}. Other non-tightness results are by Corollary~\ref{coro:hsbm-upper}.

Theorem~\ref{thm:hsbm-r56} follows directly from Theorem~\ref{thm:boht-r56} using the relationship between HSBM weak recovery and BOHT reconstruction ({\cite[Theorem 5.15]{gu2023channel}}).

Let us now discuss the proof of Theorem~\ref{thm:boht-r56}.
For $k\in \bZ_{\ge 0}$, let $M_k$ denote the channel $\sigma_\rho\mapsto (T_k, \sigma_{L_k})$.
Then $(M_k)_{k\ge 0}$ satisfies the belief propagation (BP) recursion
\begin{align}
  M_{k+1} = \BP(M_k)
\end{align}
where $\BP$ is the belief propagation operator defined as
\begin{align} \label{eqn:bp-operator-defn}
  \BP(P) = \bE_{b\sim D} \left(P^{\times (r-1)} \circ B_{r,\lambda}\right)^{\star b}.
\end{align}
(See Section~\ref{sec:boht:prelim} for the relevant information theory background.)
In other words, $M_k = \BP^k(\Id)$, where $\Id: \{\pm\}\to \{\pm\}$ denotes the identity channel.
Then non-reconstruction is equivalent to
\begin{align} \label{eqn:boht-non-recon}
  \lim_{k\to \infty} C(\BP^k(\Id)) = 0,
\end{align}
where $C$ denotes channel capacity.
Note that the channels $M_k$ are binary memoryless symmetric (BMS) channels (Section~\ref{sec:boht:prelim}) and the $\BP$ operator sends BMS channels to BMS channels.

Our proof is a hypertree version of Sly's method. This method was introduced in \cite{sly2009reconstruction,sly2011reconstruction} for reconstruction on the Potts model. It has been applied to broadcasting with binary asymmetric channels (\cite{liu2019large}) and is recently refined by \cite{mossel2023exact} for the Potts channel.
The method has two parts: large degree asymptotics and robust reconstruction.

In the first part (large degree asymptotics), we prove the following result.
\begin{proposition}[Large degree asymptotics] \label{prop:large-deg-asymp-r56}
  Fix $r\le 6$. For any $\epsilon>0$, there exists $d_0=d_0(r,\epsilon)$ such that for any $d\ge d_0$ and $\lambda\in \left[-\frac 1{2^{r-1}-1}, 1\right]$ with $(r-1) d \lambda^2 \le 1$, we have
  \begin{align} \label{eqn:large-degree-asymp-main}
    \lim_{k\to \infty} C_{\chi^2}(\BP^k(\Id)) < \epsilon,
  \end{align}
  where $C_{\chi^2}$ denotes $\chi^2$-capacity (Section~\ref{sec:boht:prelim}).
\end{proposition}
The idea is that $C_{\chi^2}(P^{\star d})$ for a BMS channel $P$ is the expectation of a function involving $d$ i.i.d.~variables. When $d$ is large, the result can be approximated using the expectation of a function involving a Gaussian variable.
For the BOHT model, \cite{gu2023weak} established the Gaussian approximation.
The resulting function is a Gaussian integral (Eq.~\eqref{eqn:thm-boht-asymp-grw}), and we prove the numerical observation that it has no fixed points on $(0, 1]$.

The second part (robust reconstruction) is more difficult. We show that for any $r\ge 2$, there exists $\epsilon=\epsilon(r)>0$ such that if $(r-1) d \lambda^2\le 1$, then for any BMS channel $P$ with $C_{\chi^2}(P) \le \epsilon$, we have
\begin{align}
  \lim_{k\to \infty} C_{\chi^2}(\BP^k(P)) = 0.
\end{align}
This is called robust reconstruction because we are considering reconstruction of the root label $\sigma_\rho$ given observations of the leaf labels $\sigma_{L_k}$ through a very noisy channel $P$, instead of the identity channel $\Id$. \cite{janson2004robust} studied the robust reconstruction problem on trees, and proved that the robust reconstruction threshold matches the KS threshold in most cases. However, their result does not apply to the hypertree model, and does not work for the critical case (i.e., at the KS threshold).
The robust reconstruction result we prove is very precise: it works at criticality and the robustness parameter $\epsilon$ does not depend on $d$.

Let $x_k = C_{\chi^2}(\BP^k(\Id))$. We prove that if $x_k$ is small enough, then
\begin{align} \label{eqn:robust-recon-main-expansion}
  x_{k+1} = (r-1) d \lambda^2 x_k - \left((r-1)(r-2) d \lambda^3 + (r-1)^2 \left(\bE_{b\sim D} b(b-1)\right) \lambda^4\right) x_k^2 + O_r(x_k^3).
\end{align}
This is the main expansion in the proof of robust reconstruction.
The proof of Eq.~\eqref{eqn:robust-recon-main-expansion} is by expanding $C_{\chi^2}(\BP(P))$ when $C_{\chi^2}(P)$ is close to $0$.

There are several technical difficulties in the this approach.
The first difficulty is that $C_{\chi^2}(\BP(P))$ is not a function of $C_{\chi^2}(P)$.
Any BMS $P$ is described by a random variable $\theta_P$ on $[-1, 1]$ (see Section~\ref{sec:boht:prelim}).
The $\chi^2$-information $C_{\chi^2}(P)$ satisfies
\begin{align}
  C_{\chi^2}(P) = \bE \theta_P = \bE \theta_P^2.
\end{align}
Then $C_{\chi^2}(\BP(P))$ is the expectation of a function of $(r-1) d$ i.i.d.~$\theta_P$ variables.
In the expansion of $C_{\chi^2}(\BP(P))$, there will be higher order moments of $\theta_P$ which a priori may affect the first few orders in Eq.~\eqref{eqn:robust-recon-main-expansion}.
Our solution is to also perform expansion on the third moment. We define
\begin{align}
  y_k = \bE \theta_{M_k}^3 = \bE \theta_{M_k}^4.
\end{align}
Because $\|\theta_P\|\le 1$ almost surely for any $P$, we have $y_k \le x_k$.
We first perform a crude expansion of $x_{k+1}$ and $y_{k+1}$, which roughly speaking allows us to conclude that $y_k = O(x_k^2)$. In this way we can control the influence of $y_k$ and perform a finer expansion of $x_{k+1}$, which give us Eq.~\eqref{eqn:robust-recon-main-expansion}.

The second difficulty is that, unlike previous works which work with trees, we work with hypertrees. As we noted, $C_{\chi^2}(\BP(P))$ is the expectation of a function with $(r-1) d$ variables, while previous works have $d$ variables. When performing the expansion, there will be a large number of terms, even after considering symmetry. Our solution is to perform the expansion in two steps.
Let $M'_{k+1}$ be the channel $M_k^{\times (r-1)} \circ B_{r,\lambda}$.
We decompose the BP recursion $M_{k+1}=\BP(M_k)$ into two conceptually smaller steps $M_k \mapsto M'_{k+1}$ and $M'_{k+1}\mapsto M_{k+1}$.
Equivalently, we decompose the expansion step $(x_k,y_k)\mapsto (x_{k+1},y_{k+1})$ into two steps $(x_k, y_k) \mapsto (x'_{k+1}, y'_{k+1})$ and $(x'_{k+1}, y'_{k+1}) \mapsto (x_{k+1},y_{k+1})$, where
\begin{align}
  x'_k = \bE \theta_{M'_k} = \bE \theta_{M'_k}^2, \qquad
  y'_k = \bE \theta_{M'_k}^3 = \bE \theta_{M'_k}^4.
\end{align}
The fine expansions we get for $x$ and $x'$ are as follows.
\begin{align}
  x'_{k+1} &= (r-1) \lambda^2 x_k - (r-1)(r-2)\lambda^3 x_k^2 + O_r(\lambda^3 x_k^3), \\
  x_{k+1} &= d x'_{k+1} - \bE_{b\sim D} b(b-1) x_{k+1}^{\prime 2} + O_r(d^3 x_{k+1}^{\prime 3}).
\end{align}
Combining these smaller steps leads to the main expansion Eq.~\eqref{eqn:robust-recon-main-expansion}.

After establishing Eq.~\eqref{eqn:robust-recon-main-expansion}, we observe that for any parameters $r\ge 3$, $\lambda\in \left[-\frac 1{2^{r-1}-1}, 1\right]$, $d\ge 0$, if $(r-1) d \lambda^2 = 1$, then the coefficient of $x_k^2$ is negative. This means that $x_k$ contracts if it is small enough.
In this way we get the robust reconstruction result.
\begin{theorem}[Robust reconstruction is impossible below and at KS] \label{thm:boht-robust}
  Fix $r\ge 3$.
  Consider the model $\BOHT(r,\lambda,d)$ or $\BOHT(r,\lambda,\Pois(d))$.
  There exists $\epsilon=\epsilon(r)>0$ such that if $(r-1) d \lambda^2 \le 1$, then for any BMS channel $P$ with $C_{\chi^2}(P) \le \epsilon$, we have
  \begin{align}
    \lim_{k\to \infty} C_{\chi^2}(\BP^k(P)) = 0.
  \end{align}
  In other words, for fixed $r$, robust reconstruction is impossible below and at the KS threshold in a way that is uniform in $d$ and $\lambda$.
\end{theorem}

Finally, Theorem~\ref{thm:boht-r56} is proved by combining Prop.~\ref{prop:large-deg-asymp-r56} and Theorem~\ref{thm:boht-robust}. We first take $\epsilon>0$ in the robust reconstruction part such that Theorem~\ref{thm:boht-robust} holds. Then we take $d$ large enough such that large degree asymptotics works and Eq.~\eqref{eqn:large-degree-asymp-main} holds.


Our proof of Theorem~\ref{thm:boht-r56} also determines the asymptotic reconstruction threshold for BOHT with $r\ge 7$.
\begin{theorem}[Asymptotic BOHT reconstruction threshold] \label{thm:boht-asymp}
  Consider the model $\BOHT(r,\lambda,d)$ or $\BOHT(r,\lambda,\Pois(d))$ where $r\ge 7$ and $\lambda\in \left[-\frac 1{2^{r-1}-1}, 1\right]$.
  Let $w^*=w^*(r)$ be defined as
  \begin{align}
    w^*(r) &= \sup\{w > 0: g_{r,w}(x) \le x \forall 0\le x\le 1\}, \label{eqn:thm-boht-asymp-wstar}\\
    g_{r,w}(x) &= \bE_{Z\sim \cN(0,1)} \tanh\left( s_{r,w}(x) + \sqrt{s_{r,w}(x)} Z\right), \label{eqn:thm-boht-asymp-grw} \\
    s_{r,w}(x) &= \frac w2 \left((1+x)^{r-1}-(1-x)^{r-1}\right). \label{eqn:thm-boht-asymp-srw}
  \end{align}
  Then for any $\epsilon>0$, there exists $d_0 = d_0(r,\epsilon)$ such that for $d\ge d_0$, if $d \lambda^2 \ge w^*+\epsilon$ then reconstruction is possible; if $d \lambda^2 \le w^*-\epsilon$ then reconstruction is impossible.
\end{theorem}

\subsection{Organization} \label{sec:intro:org}
In Section~\ref{sec:boht} we present the proof structure for Theorem~\ref{thm:boht-r56} (tightness of the KS threshold for $r\le 6$ and large $d$), which includes the two main parts Prop.~\ref{prop:large-deg-asymp-r56} (large degree asymptotics) and Theorem~\ref{thm:boht-robust} (robust reconstruction below and at KS).
In Section~\ref{sec:hsbm-upper} we prove Theorem~\ref{thm:hsbm-upper} (non-asymptotic HSBM weak recovery upper bound).

In Section~\ref{sec:boht-app} we provide missing proofs of lemmas used in Section~\ref{sec:boht}.
In Section~\ref{sec:remain} we provide proofs of our remaining main results, including Corollary~\ref{coro:hsbm-upper} (non-tightness of the KS threshold), Theorem~\ref{thm:hsbm-asymp} (asymptotic HSBM weak recovery bounds), and Theorem~\ref{thm:boht-asymp} (asymptotic BOHT reconstruction threshold).
In Section~\ref{sec:discuss} we make discussions on related works and further directions.

\section{Reconstruction on hypertrees} \label{sec:boht}

\subsection{Preliminaries} \label{sec:boht:prelim}
A binary memoryless symmetric (BMS) channel is a channel $P: \{\pm\}\to \cY$ together with a measurable involution $\sigma: \cY \to \cY$ such that $P(E|+)=P(\sigma(E)|-)$ for all measurable $E\subseteq \cY$.
Examples of BMS channels include binary erasure channels (BECs) and binary symmetric channels (BSCs). The broadcast channel $B_{r,\lambda}$ (Eq.~\eqref{eqn:broadcast-channel}) and the channels $M_k$ (defined as $\sigma_\rho \mapsto (T_k, \sigma_{L_k})$) are all BMS channels (with the involution being the global sign flip).

Every BMS channel is equivalent to a mixture of BSC channels. That is, every BMS channel $P$ is equivalent to a channel $X\mapsto (\Delta,Z)$ where $\Delta\in \left[0,\frac 12\right]$ is independent of $X$ and $P_{Z|\Delta,X}=\BSC_\Delta(\cdot|X)$. We use $\Delta_P$ to denote the $\Delta$-component of $P$. It is a random variable taking values in the interval $\left[0,\frac 12\right]$.
In this way, BMS channels (up to channel equivalence) are in one-to-one correspondence with distributions on $\left[0,\frac 12\right]$.

In computations, it is often easier to work with the signed $\theta$-component, rather than the $\Delta$-component.
Let $P$ be a BMS channel and $\Delta_P$ be its $\Delta$-component.
Define $\theta_P\in [-1,1]$ to be the random variable such that conditioned on $\Delta_P$,
\begin{align}
  \theta_P=\left\{
    \begin{array}{ll}
      1-2\Delta_P, & \text{w.p.}~1-\Delta_P,\\
      2\Delta_P-1, & \text{w.p.}~\Delta_P.
    \end{array}
  \right.
\end{align}
We call $\theta_P$ the signed $\theta$-component of $P$.

The $\chi^2$-capacity of a BMS channel $P$ is defined as
\begin{align}
  C_{\chi^2}(P) &= I_{\chi^2}(\pi, P) = \bE (1-2\Delta_P)^2 = \bE \theta_P = \bE \theta_P^2,
\end{align}
where $\pi = \Unif(\{\pm\})$. For BMS channels, the $\chi^2$-capacity is within a constant factor of the capacity. That is,
\begin{align}
  \frac 12 C_{\chi^2}(P)\le C(P) \le (\log 2) C_{\chi^2}(P).
\end{align}
Therefore, Eq.~\eqref{eqn:boht-non-recon} holds if and only if
\begin{align}
  \lim_{k\to \infty} C_{\chi^2}(\BP^k(\Id)) = 0.
\end{align}

Let $P: \cX\to \cY$ and $Q: \cX\to \cZ$ be two channels with the same input alphabet. We define the $\star$-convolution channel $P\star Q: \cX \to \cY\times \cZ$ as $(P\star Q)(E\times F|x) = P(E|x) Q(F|x)$ for measurable subsets $E\subseteq \cY$, $F\subseteq \cZ$. For $n\in \bZ_{\ge 0}$, we use $P^{\star n}$ to denote the $n$-th $\star$-convolution power.

Let $P: \cX\to \cY$ and $P': \cX'\to \cY'$ be two channels. We define the tensor product channel $P\times P': \cX\times \cX'\to \cY\times \cY'$ as $(P\times P')(E\times E'|(x,x')) = P(E|x) P(E'|x')$ for measurable subsets $E\subseteq \cX$, $E'\subseteq \cX'$. For $n\in \bZ_{\ge 0}$, we use $P^{\times n}$ to denote the $n$-th tensor power.

Let $\cA$ be a collection of information channels with the same input alphabet $\cX$ and $\mu$ be a distribution on $\cA$. We define $\bE_{P\sim \mu} P$ to be the channel that maps $x\in \cX$ to $(P, Y)$, where $P\sim \mu$ and $Y\sim P(\cdot|x)$. This channel is called a mixture of $\cA$.

With the above definitions, the BP operator (Eq.~\eqref{eqn:bp-operator-defn}) is well-defined and maps BMS channels to BMS channels.

\subsection{Large degree asymptotics} \label{sec:boht:large-deg}
\begin{proposition}[Large degree asymptotics, {\cite[Prop.~30]{gu2023weak}}] \label{prop:large-deg-asymp}
  Fix $r\ge 2$. For any $\epsilon>0$ and $C>0$, there exists $d_0=d_0(r,\epsilon)$ such that for any $d\ge d_0$ and $\lambda\in \left[-\frac 1{2^{r-1}-1}, 1\right]$ with $(r-1) d \lambda^2 \le C$, for any BMS channel $P$ we have
  \begin{align}
    \left|C_{\chi^2}(\BP(P)) - g_{r,d\lambda^2}(C_{\chi^2}(P))\right| \le \epsilon,
  \end{align}
  where $g_{r,w}(x)$ is defined in Eq.~\eqref{eqn:thm-boht-asymp-grw}.
\end{proposition}
\cite[Prop.~30]{gu2023weak} stated the result only for $\lambda>0$ and $C=1$, but their proof works without any change for $\lambda\in \left[-\frac 1{2^{r-1}-1}, 1\right]$ and constant $C$.
\begin{lemma}[Properties of $g$] \label{lem:large-step-g-property}
  The function $g_{r,w}$ and $s_{r,w}$ defined in Eq.~\eqref{eqn:thm-boht-asymp-grw} satisfies the following properties.
  \begin{enumerate}[label=(\roman*)]
    \item \label{item:lem-large-step-g-property:i} $g_{r,w}(x)$ is non-decreasing in $r\in \bZ_{\ge 2}$ (strict when $w\ne 0$ and $x\ne 0$), $w\in \bR_{\ge 0}$ (strict when $x\ne 0$), and $x\in [0, 1]$ (strict when $w\ne 0$).
    \item \label{item:lem-large-step-g-property:ii} $g_{r,w}(x) \le s_{r,w}(x)$. The inequality is strict when $w\ne 0$ and $x\ne 0$.
  \end{enumerate}
\end{lemma}
\begin{lemma} \label{lem:large-step-g-contract-r56}
  Fix $r\le 6$. Then $g_{r,\frac 1{r-1}}(x) < x$ for all $0<x\le 1$.
\end{lemma}
Proofs of Lemma~\ref{lem:large-step-g-property} and Lemma~\ref{lem:large-step-g-contract-r56} are deferred to Section~\ref{sec:boht-app:large-deg-asymp}.
\begin{proof}[Proof of Prop.~\ref{prop:large-deg-asymp-r56}]
  By Lemma~\ref{lem:large-step-g-contract-r56}, for any $\epsilon>0$ there exists $\delta>0$ such that $g_{r,\frac 1{r-1}}(x) \le x-\delta$ for all $x\in [\epsilon,1]$.
  Take $d_0=d_0(r,\delta/2)$ in Prop.~\ref{prop:large-deg-asymp}.
  Then for any BMS channel $P$ with $C_{\chi^2}(P) \ge \epsilon$, we have
  \begin{align}
    C_{\chi^2}(\BP(P)) \le g_{r,d\lambda^2}(C_{\chi^2}(P)) + \delta/2 \le C_{\chi^2}(P) - \delta/2.
  \end{align}
  If Eq.~\eqref{eqn:large-degree-asymp-main} does not hold, then
  \begin{align}
    \lim_{k\to \infty} C_{\chi^2}(M_k) = \lim_{k\to \infty} C_{\chi^2}(M_{k+1}) \le \lim_{k\to \infty} (C_{\chi^2}(M_k)-\delta/2),
  \end{align}
  which is absurd.
\end{proof}

\subsection{Robust reconstruction} \label{sec:boht:robust}
In this section we prove Theorem~\ref{thm:boht-robust}.
We note that the case $\lambda\ge \frac 15$ follows from \cite[Theorem 2(iii)]{gu2023weak}.
In the following, assume that $\lambda\in \left[-\frac 1{2^{r-1}-1}, \frac 15\right]$.

Fix $d$, $\lambda$, and $P$ as in the statement of Theorem~\ref{thm:boht-robust}.
Our goal is to prove that
\begin{align}
  \lim_{k\to \infty} C_{\chi^2}(\BP^k(P)) = 0
\end{align}
when $C_{\chi^2}(P)$ is sufficiently small.

For $k\in \bZ_{\ge 0}$, define
\begin{align}
  M_{P,k} = \BP^k(P), \qquad M'_{P,k+1} = M_{P,k}^{\times (r-1)} \circ B_{r,\lambda}.
\end{align}

Define
\begin{align}
  x_k &= C_{\chi^2}(M_{P,k}) = \bE \theta_{M_{P,k}} = \bE \theta_{M_{P,k}}^2, \\
  y_k &= \bE \theta_{M_{P,k}}^3 = \bE \theta_{M_{P,k}}^4, \\
  x'_{k+1} &= C_{\chi^2}(M'_{P,k+1}) = \bE \theta_{M'_{P,k+1}} = \bE \theta_{M'_{P,k+1}}^2, \\
  y'_{k+1} &= \bE \theta_{M'_{P,k+1}}^3 = \bE \theta_{M'_{P,k+1}}^4.
\end{align}
Because $|\theta|\le 1$, we have $y_k\le x_k$, $y'_{k+1}\le x'_{k+1}$ for all $k\in \bZ_{\ge 0}$.

\begin{lemma}[A priori estimates] \label{lem:boht-robust:a-priori}
  We have
  \begin{align}
    \label{eqn:lem-boht-robust-a-priori:i} x'_{k+1} &= (r-1)\lambda^2 x_k + O_r(\lambda^2 x_k^2), \\
    \label{eqn:lem-boht-robust-a-priori:ii} y'_{k+1} &= (r-1)\lambda^4 y_k + O_r(\lambda^2 x_k^2), \\
    \label{eqn:lem-boht-robust-a-priori:iii} x_{k+1} &= d x'_{k+1} + O_r(d^2 x_{k+1}^{\prime 2}), \\
    \label{eqn:lem-boht-robust-a-priori:iv} y_{k+1} &= d y'_{k+1} + O_r(d^2 x_{k+1}^{\prime 2}).
  \end{align}
  In particular, we have
  \begin{align}
    \label{eqn:lem-boht-robust-a-priori:final-i} x_{k+1} &= (r-1) d \lambda^2 x_k + O_r(x_k^2),\\
    \label{eqn:lem-boht-robust-a-priori:final-ii} y_{k+1} &= (r-1) d \lambda^4 y_k + O_r(x_k^2).
  \end{align}
\end{lemma}
Proof of Lemma~\ref{lem:boht-robust:a-priori} is deferred to Section~\ref{sec:boht-app:a-priori}.

Lemma~\ref{lem:boht-robust:a-priori} already implies that robust reconstruction is impossible for $(r-1) d \lambda^2 < 1$. However the robustness parameter $\epsilon$ goes to $0$ as $(r-1) d \lambda^2 \to 1$.
Nevertheless, we can obtain that there exists $\epsilon=\epsilon(r)>0$ such that for $(r-1) d \lambda^2 \le 0.99$, robust reconstruction is impossible with respect to any BMS channel $P$ with $C_{\chi^2}(P) \le \epsilon$.
In the following, assume that $0.99 \le (r-1) d \lambda^2 \le 1$.

Let $\gamma=\gamma(r)>0$ be a large enough constant.
We say a BMS channel $P$ is $\gamma$-normal if
\begin{align}
  \bE \theta_P^4 \le \gamma \left(\bE \theta_P^2\right)^2.
\end{align}
Using Lemma~\ref{lem:boht-robust:a-priori}, we can show that $M_{P,k}$ is $\gamma$-normal for some $k$ not too large.
\begin{lemma} \label{lem:boht-robust:normal}
  For large enough $\gamma=\gamma(r)$ the following holds.
  For any $\epsilon>0$, there exists $\delta=\delta(r,\epsilon)>0$, such that for any BMS channel $P$ with $C_{\chi^2}(P) \le \delta$, for some $k_0\in \bZ_{\ge 0}$, we have
  \begin{align}
    x_{k_0} \le \epsilon, \qquad y_{k_0} \le \gamma x_{k_0}^2.
  \end{align}
\end{lemma}
Proof of Lemma~\ref{lem:boht-robust:normal} is deferred to Section~\ref{sec:boht-app:a-priori}.

Using the $\gamma$-normal property, we can obtain finer estimates.
\begin{lemma}[Fine estimates] \label{lem:boht-robust:fine}
  For large enough $\gamma=\gamma(r)$ the following holds.
  If $M_{P,k}$ is $\gamma$-normal, then
  \begin{align}
    \label{eqn:lem-boht-robust-fine:i} x'_{k+1} &= (r-1)\lambda^2 x_k - (r-1)(r-2) \lambda^3 x_k^2+ O_r(\lambda^2 x_k^3),\\
    \label{eqn:lem-boht-robust-fine:ii} x_{k+1} &= d x'_{k+1} - \left(\bE_{b\sim d} b(b-1)\right) x_{k+1}^{\prime 2} + O_r(d^3 x_{k+1}^{\prime 3}).
  \end{align}
  In particular,
  \begin{align} \label{eqn:lem-boht-robust-fine:final}
    x_{k+1} &= (r-1) d \lambda^2 x_k -  ((r-1)(r-2)d\lambda^3 + (r-1)^2\left(\bE_{b\sim D} b(b-1)\right) \lambda^4) x_k^2+ O_r(x_k^3).
  \end{align}
  Furthermore, $M_{P,k+1}$ is $\gamma$-normal.
\end{lemma}
Proof of Lemma~\ref{lem:boht-robust:fine} is deferred to Section~\ref{sec:boht-app:fine}.

\begin{proof}[Proof of Theorem~\ref{thm:boht-robust}]

  If $\lambda\ge \frac 15$, then the result follows from the reconstruction threshold for the BOHT model (\cite[Theorem 2(iii)]{gu2023weak}). In the following assume $\lambda\in \left[-\frac 1{2^{r-1}-1}, \frac 15\right]$.

  If $(r-1) d \lambda^2 \le 0.99$, then the result follows from Eq.~\eqref{eqn:lem-boht-robust-a-priori:final-i} by taking $\epsilon>0$ small enough. In the following assume $0.99 \le (r-1) d \lambda^2 \le 1$.

  In Eq.~\eqref{eqn:lem-boht-robust-fine:final}, we have
  \begin{align} \label{eqn:proof-boht-robust:step-1}
    &~(r-1)(r-2)d\lambda^3 + (r-1)^2\left(\bE_{b\sim D} b(b-1)\right) \lambda^4 \\
    \nonumber \ge&~  (r-1)(r-2)d\lambda^3 + (r-1)^2 d (d-1) \lambda^4 \\
    \nonumber =&~ (r-1) d \lambda^2 ((r-1) d \lambda^2 + (r-2) \lambda - (r-1) \lambda^2).
  \end{align}
  For $\lambda\in \left[-\frac 1{2^{r-1}-1},0\right]$, we have
  \begin{align}
    \eqref{eqn:proof-boht-robust:step-1} \ge 0.99 \cdot \left(0.99 - \frac{r-2}{2^{r-1}-1} - \frac{r-1}{(2^{r-1}-1)^2}\right) \ge 0.4.
  \end{align}
  For $\lambda\in \left[0,\frac 15\right]$, we have
  \begin{align}
    \eqref{eqn:proof-boht-robust:step-1} \ge 0.99 \cdot \left(0.99 + \lambda ((r-2)-(r-1)\lambda)\right) \ge 0.99^2.
  \end{align}
  Therefore $\eqref{eqn:proof-boht-robust:step-1} \ge 0.4$ always holds.

  Take $\gamma$ large enough such that Lemma~\ref{lem:boht-robust:normal} and Lemma~\ref{lem:boht-robust:fine} hold.
  Take $\epsilon$ such that Eq.~\eqref{eqn:lem-boht-robust-fine:final} has no fixed point for all $x\le \epsilon$. Such $\epsilon$ exists because the coefficient of $x^2$ is upper bounded by $-0.4$.
  Then for any $\gamma$-normal BMS channel $P$ with $C_{\chi^2}(P) \le \epsilon$, robust reconstruction with respect to $P$ is impossible.

  Take $\delta$ such that Lemma~\ref{lem:boht-robust:normal} holds.
  Then for any BMS channel $P$ with $C_{\chi^2}(P)\le \delta$, by Lemma~\ref{lem:boht-robust:normal}, for some $k_0\in \bZ_{\ge 0}$, $M_{k_0,P}$ is $\gamma$-normal and $C_{\chi^2}(M_{k_0,P}) \le \epsilon$, so robust reconstruction with respect to $M_{k_0,P}$, thus $P$, is impossible.
  This finishes the proof.
\end{proof}

\subsection{Finishing the proof} \label{sec:boht:main}
We are now ready to prove Theorem~\ref{thm:boht-r56}.
\begin{proof}[Proof of Theorem~\ref{thm:boht-r56}]
  Take $\epsilon=\epsilon(r)$ in Theorem~\ref{thm:boht-robust}.
  Take $d_0=d_0(r,\epsilon)$ in Prop.~\ref{prop:large-deg-asymp-r56}.
  Fix $d\ge d_0$, $\lambda\in \left[-\frac 1{2^{r-1}-1}, 1\right]$ such that $(r-1) d \lambda^2 \le 1$.
  By Prop.~\ref{prop:large-deg-asymp-r56}, there exists $k_0$ such that $C_{\chi^2}(M_{k_0}) \le \epsilon$.
  By Theorem~\ref{thm:boht-robust}, we have
  \begin{align}
    \lim_{k\to \infty} C_{\chi^2}(M_k)
    = \lim_{k\to \infty} C_{\chi^2}(M_{k+k_0})
    = \lim_{k\to \infty} C_{\chi^2}(\BP^k(M_{k_0}))
    = 0.
  \end{align}
\end{proof}

\section{HSBM weak recovery upper bound} \label{sec:hsbm-upper}
In this section we prove Theorem~\ref{thm:hsbm-upper}.
We say an event happens a.a.s.~(asymptotically almost surely) if it happens with probability $1-o(1)$ as $n\to \infty$.

\begin{proof}[Proof of Theorem~\ref{thm:hsbm-upper}]
  We first describe our algorithm. Suppose we are given a HSBM instance $G=(V,E)$.
  We say a partition $Y: V \to \{\pm\}$ is almost balanced if $\left|\# Y^{-1}(+) - \frac n2 \right| \le n^{0.6}$.
  We say a hyperedge $e\in E$ is in-community with respect to $Y$ if $Y_v$s are equal for all $v\in e$.
  We say a partition $Y: V \to \{\pm\}$ is good if the number of in-community hyperedges is in $\left[\frac{an}{2^{r-1} r} - n^{0.6}, \frac{an}{2^{r-1} r} + n^{0.6}\right]$, and bad otherwise.
  The algorithm looks for a good almost balanced partition. If such a partition exists, return any of them; otherwise report failure.

  We say a partition $Y$ is correlated with $X$ (the true partition) if $d_H(Y, X) \le \left(\frac 12-\epsilon\right) n$, and uncorrelated with $X$ otherwise.
  We prove that for some $\epsilon>0$, a.a.s.~the following are true.
  \begin{enumerate}[label=(\alph*)]
    \item \label{item:proof-thm-hsbm-upper:i} There exists a good almost balanced partition.
    \item \label{item:proof-thm-hsbm-upper:ii} All almost balanced partitions that are uncorrelated with $X$ are bad.
  \end{enumerate}
  If \ref{item:proof-thm-hsbm-upper:i} holds, our algorithm returns a good partition. If in addition \ref{item:proof-thm-hsbm-upper:ii} holds, the returned partition is correlated with the true partition $X$, and the algorithm succeeds.
  It remains to prove \ref{item:proof-thm-hsbm-upper:i} and \ref{item:proof-thm-hsbm-upper:ii}.

  \paragraph{Proof of \ref{item:proof-thm-hsbm-upper:i}}
  We show that $X$ is a.a.s.~a good almost balanced partition.
  By the central limit theorem, $X$ is a.a.s.~almost balanced.
  For such $X$, the number of in-community sets $S\in \binom{V}r$ is $\binom{|X|}r+\binom{n-|X|}r = (1\pm O(n^{-1}))\frac{n^r}{2^{r-1}\cdot r!}$. Each such set is an hyperedge independently with probability $\frac a{\binom{n}{r-1}} = (1\pm O(n^{-1}))\frac{a (r-1)!}{n^{r-1}}$.
  By Chebyshev's inequality, a.a.s.~the number of in-community hyperedges is in $\left[\frac{an}{2^{r-1} r} - n^{0.6}, \frac{an}{2^{r-1} r} + n^{0.6}\right]$.

  \paragraph{Proof of \ref{item:proof-thm-hsbm-upper:ii}}
  By ignoring events that happen with probability $o(1)$, we make the following assumptions.
  \begin{enumerate}
    \item $X$ is almost balanced.
    \item The total number of hyperedges $m$ is in $\left[ \frac{dn}r - n^{0.6}, \frac {dn}r - n^{0.6}\right]$.
  \end{enumerate}

  We show that \ref{item:proof-thm-hsbm-upper:ii} is true conditioned on $m$.
  We consider the following random hypergraph model $G(n,m,X)$. For each of the $m$ hyperedge, generate a sequence $(x_1,\ldots,x_r)$ of $r$ bits according to the following distribution
  \begin{align}
    p_{x_1,\ldots,x_r} = \left\{
      \begin{array}{ll}
        \frac 12 \cdot \frac{a}{a-b+2^{r-1} b}, & \text{if}~x_1=\cdots=x_r,\\
        \frac 12 \cdot \frac{b}{a-b+2^{r-1} b}, & \text{o.w.}
      \end{array}
    \right.
  \end{align}
  Then choose $r$ vertices $v_1,\ldots,v_r$ independently according to $v_i\sim \Unif(X^{-1}(x_i))$, and add hyperedge $e=(v_1,\ldots,v_r)$ to the hypergraph.

  The hypergraph $G(n,m,X)$ is not necessarily simple and may have degenerate hyperedges (i.e., hyperedges with less than $r$ distinct vertices). However, with constant probability, this does not happen. Furthermore, conditioned on that $G(n,m,X)$ is simple, it has the same distribution as $\HSBM(n,r,a,b)$ conditioned on $X$ and $m$.
  Therefore, any event that happens a.a.s.~for $G(n,m,X)$ also happens a.a.s.~for $\HSBM(n,r,a,b)$ conditioned on $X$ and $m$.

  Now consider an almost balanced partition $Y$ uncorrelated with $X$.
  Say $d_H(Y, X) = \left(\frac 12-\delta\right)n$ for some $\delta \le \epsilon$.
  By possibly replacing $Y$ with $-Y$, WLOG assume that $\#\{v\in V: X_v\ne Y_v\}=\left(\frac 12-\delta\right)n$.
  For $x,y\in \{\pm\}$, let $c_{x,y}$ denote the number of vertices $v\in V$ such that $X_v=x$ and $Y_v=y$.
  Then
  \begin{align}
    &c_{++} = \left(\frac 14 + \frac{\delta}2\right)n \pm O(n^{0.6}),
    &c_{+-} = \left(\frac 14 - \frac{\delta}2\right)n \pm O(n^{0.6}), \\
    \nonumber &c_{-+} = \left(\frac 14 - \frac{\delta}2\right)n \pm O(n^{0.6}),
    &c_{--} = \left(\frac 14 + \frac{\delta}2\right)n \pm O(n^{0.6}).
  \end{align}
  Consider the process of generating one hyperedge under $G(n,m,X)$.
  Let $y_i = Y_{v_i}$. Then we have
  \begin{align}
    \bP[y_i=+|x_i=+] = \frac 12 + \delta \pm O(n^{-0.4}), \qquad
    \bP[y_i=+|x_i=-] = \frac 12 - \delta \pm O(n^{-0.4}).
  \end{align}
  Therefore, under $G(n,m,X)$, each hyperedge is in-community with respect to $Y$ with probability
  \begin{align}
    &~\sum_{x_1,\ldots,x_r\in \{\pm\}} p_{x_1,\ldots,x_r} \left(
      \prod_{i\in [r]} \bP[y_i=+|x_i]
      + \prod_{i\in [r]} \bP[y_i=-|x_i]
    \right) \\
    \nonumber =&~ \sum_{0\le j\le r} p_{+^j -^{r-j}} \binom rj \left(\left( \frac 12 + \delta\right)^j \left( \frac 12 - \delta\right)^{r-j} + \left( \frac 12 - \delta\right)^j \left( \frac 12 + \delta\right)^{r-j}\right) \pm O(n^{-0.4}) \\
    \nonumber  =&~ \frac{1-\lambda}{2^{r-1}} + \lambda \left( \left(\frac 12 + \delta\right)^r + \left(\frac 12 - \delta\right)^r \right) \pm O(n^{-0.4}).
  \end{align}
  The last value can be arbitrarily close to $\frac 1{2^{r-1}}$ for small enough $\delta>0$.

  By Sanov's theorem, the probability that there are $\left[\frac{an}{2^{r-1} r} - n^{0.6}, \frac{an}{2^{r-1} r} + n^{0.6}\right]$ in-community hyperedges with respect to $Y$ is
  \begin{align}
    \exp\left( -\frac{n d}{r} \cdot d_{\KL}\left(\lambda + \frac{1-\lambda}{2^{r-1}} \left\| \frac{1-\lambda}{2^{r-1}} + \lambda \left( \left(\frac 12 + \delta\right)^r + \left(\frac 12 - \delta\right)^r \right) \right. \right) \pm o(n)\right).
  \end{align}
  By Eq.~\eqref{eqn:thm-hsbm-upper:cond}, for $\epsilon>0$ close enough to $0$, the last expression is
  $
    \exp\left(- n (\log 2 + \epsilon') \pm o(n) \right)
  $
  for some $\epsilon'>0$.
  Therefore every uncorrelated balanced partition $Y$ is good with probability $\exp\left(- n (\log 2 + \epsilon') \pm o(n) \right)$.
  On the other hand, the total number of partitions is $2^n$.
  By union bound, a.a.s.~there are no good uncorrelated balanced partitions.
  This finishes the proof of \ref{item:proof-thm-hsbm-upper:ii}.
\end{proof}

\ifdefined\isarxiv
\section*{Acknowledgments}
Y.G.~is supported by the National Science Foundation under Grant No. DMS-1926686. We thank anonymous reviewers for helpful comments and suggestions.
\else
\acks{Y.G.~is supported by the National Science Foundation under Grant No. DMS-1926686. We thank anonymous reviewers for helpful comments and suggestions.}
\fi

\ifdefined\isarxiv
\bibliographystyle{alpha}
\fi
\bibliography{ref}

\appendix

\section{Reconstruction on hypertrees} \label{sec:boht-app}

\subsection{More preliminaries} \label{sec:boht-app:prelim}
Let $P: \cX \to \cY$ and $Q: \cX \to \cZ$ be two channels.
We say $P$ is more degraded than $Q$, denoted $P\le_{\deg} Q$, if there exists a channel $R: \cZ \to \cY$ such that $P=R\circ Q$.
If $P$ and $Q$ are BMSs, then it is possible to take $R$ such that it respects the involution structure of BMSs.

For two BMS channels $P$ and $Q$, $P\le_{\deg} Q$ if and only if there exists a coupling between $\Delta_P$ and $\Delta_Q$ such that
$
  \bE[\Delta_Q | \Delta_P] \le \Delta_P
$
almost surely.

For two BMS channels $P$ and $Q$, if $P\le_{\deg} Q$, then
\begin{align}
  C_{\chi^2}(P) \le C_{\chi^2}(Q).
\end{align}
This is known as the data processing inequality (DPI).

Degradation relationship is preserved under pre-composition, $\star$-convolution, tensor product, and mixtures.

\begin{lemma} \label{lem:boht-monotone}
  Consider two BOHT models $\BOHT(r,\lambda,D)$ and $\BOHT(r,c \lambda,D)$ for some $c\in [0,1]$.
  Let $\BP$ and $\wt \BP$ be the corresponding $\BP$ operators.
  Then for any BMS channel $P$, we have $\wt \BP(P) \le_{\deg} \BP(P)$.
  In particular, we have
  \begin{align}
    \limsup_{k\to \infty} C_{\chi^2}(\BP^k(P)) \ge \limsup_{k\to \infty} C_{\chi^2}(\wt \BP^k(P)).
  \end{align}
  In other words, if robust reconstruction with respect to $P$ is impossible for $\BOHT(r,\lambda,D)$, then robust reconstruction with respect to $P$ is impossible for $\BOHT(r,c \lambda,D)$.
\end{lemma}
\begin{proof}
  Recall that
  \begin{align}
    \BP(P) = \bE_{b\sim D} (P^{\times (r-1)} \circ B_{r,\lambda})^{\star b},\qquad \wt \BP(P) = \bE_{b\sim D} (P^{\times (r-1)} \circ B_{r,c \lambda})^{\star b}.
  \end{align}
  Because $\le_{\deg}$ is preserved under mixtures and $\star$-convolution, it suffices to prove that
  \begin{align}
    P^{\times (r-1)} \circ B_{r,c \lambda} \le_{\deg} P^{\times (r-1)} \circ B_{r,\lambda}.
  \end{align}
  Let $\cY$ be the output alphabet of $P$. Then the output alphabets of $P^{\times (r-1)} \circ B_{r,\lambda}$ and $P^{\times (r-1)} \circ B_{r,c \lambda}$ are $\cY^{r-1}$.
  We construct a channel $R: \cY^{r-1} \to \cY^{r-1}$ such that
  \begin{align} \label{eqn:proof-lem-boht-monotone:deg}
    P^{\times (r-1)} \circ B_{r, c\lambda} = R \circ P^{\times (r-1)} \circ B_{r,\lambda}.
  \end{align}
  Let $R$ be the following channel: on input $y^{r-1}\in \cY^{r-1}$, it outputs $y^{r-1}$ with probability $c$, and outputs a sample from $(P\circ \Unif(\{\pm\}))^{\times (r-1)}$ with probability $1-c$.
  It is easy to check that Eq.~\eqref{eqn:proof-lem-boht-monotone:deg} holds.
\end{proof}

Let $P: \{\pm\} \to \cY$ be a BMS channel with a discrete target space $\cY$. For $x\in \{\pm\}$, we use $P_x$ to denote the distribution $P(\cdot | x)$.
The $\theta$-component $\theta_P$ is directly related with the distribution $P_+$.
For $y\in \cY$, define
\begin{align}
  \theta_P(y)=\frac{P_+(y)-P_-(y)}{P_+(y)+P_-(y)}.
\end{align}
Then the distribution of $\theta_P(y)$ for $y\sim P_+$ is exactly the same as the distribution of $\theta_P$.
In particular,
\begin{align}
  C_{\chi^2}(P) &= \bE_{y\sim P_+} \theta_P(y).
\end{align}

\subsection{Large degree asymptotics} \label{sec:boht-app:large-deg-asymp}

\begin{proof}[Proof of Lemma~\ref{lem:large-step-g-property}]
  Let $g: [0, 1]\to [0, 1]$ be the function
  \begin{align} \label{eqn:large-step:g}
    g(s) = \bE_{Z\sim \cN(0,1)} \tanh(s + \sqrt s Z).
  \end{align}
  By \cite[Lemma 4.4]{sly2011reconstruction}, $g$ is strictly increasing and continuously differentiable on $[0, 1]$.

  \paragraph{Proof of \ref{item:lem-large-step-g-property:i}}
  Because $g$ is strictly increasing, the result follows from that $s_{r,w}(x)$ (Eq.~\eqref{eqn:thm-boht-asymp-srw}) is non-decreasing in $r\in \bZ_{\ge 2}$ (strict when $w\ne 0$ and $x\ne 0$), $w\in \bR_{\ge 0}$ (strict when $x\ne 0$), and $x\in [0, 1]$ (strict when $w\ne 0$).

  \paragraph{Proof of \ref{item:lem-large-step-g-property:ii}}
  It suffices to prove that $g(s) \le s$ for all $s\in [0, 1]$, and the inequality is strict when $s\ne 0$.
  We have
  \begin{align}
    g(s) &= \bE_{Z\sim \cN(0,1)} \tanh(s + \sqrt s Z) \\
    \nonumber &= \bE_{Z\sim \cN(0,1)} \frac 12 \left(\tanh(s + \sqrt s Z) + \tanh(s - \sqrt s Z) \right) \\
    \nonumber &\le \tanh(s) \le s.
  \end{align}
  The third step is by Lemma~\ref{lem:tanh-technical}.
  The last step is strict when $s\ne 0$.
\end{proof}

\begin{lemma} \label{lem:tanh-technical}
  For $x,y\in \bR$, if $x+y\ge 0$, then $\tanh x + \tanh y \le 2 \tanh \frac{x+y}2$.
\end{lemma}
\begin{proof}
  We have
  \begin{align}
    \tanh x + \tanh y = 2\tanh \frac{x+y}2 \cdot \frac{1+\tanh x \tanh y}{1+\tanh^2 \frac{x+y}2}.
  \end{align}
  It suffices to prove that $\tanh x \tanh y \le \tanh^2 \frac{x+y}2$.
  If $x y<0$ then LHS is $<0$ and the inequality holds.
  If $x,y\ge 0$, then the inequality follows because $\log \tanh x$ is concave on $\bR_{>0}$,
\end{proof}

\begin{proof}[Proof of Lemma~\ref{lem:large-step-g-contract-r56}]
  Recall function $g$ defined in Eq.~\eqref{eqn:large-step:g}.
  Because $g$ is strictly increasing on $[0, 1]$ and $s_{r,\frac 1{r-1}}(x)$ is non-decreasing in $r\in \{2,3,4,5,6\}$ for $0\le x\le 1$, $g_{r,\frac 1{r-1}}(x)=g\left(s_{r,\frac 1{r-1}}(x)\right)$ is non-decreasing in the same range.
  Therefore it suffices to prove the statement for $r=6$.

  Our goal is to show that
  \begin{align}
    g\left( x + 2x^3 + \frac{x^5}5 \right) < x
  \end{align}
  for all $x\in (0, 1]$.

  Let $\delta = 0.1$. We prove the result for $x\in (0,\delta]$ and $x\in [\delta, 1]$ separately.

  \textbf{Case $x\in (0,\delta]$.}
  Write $s = x + 2x^3 + \frac{x^5}5$.
  For $Z\sim \cN(0, 1)$, let $Z_+ = \exp(s+\sqrt s Z)$, $Z_- = \exp(-s-\sqrt s Z)$.
  Then
  \begin{align}
    g(s) = \bE_Z \frac{Z_+-Z_-}{Z_++Z_-} = \frac 12 \bE_Z \frac{Z_+-Z_-}{1+\frac 12 (Z_++Z_--2)}.
  \end{align}
  Expanding the last term using
  \begin{align}
    \frac 1{1+x}=1-x+x^2-x^3+\frac{x^4}{1+x},
  \end{align}
  we get
  \begin{align}
    g(s) = \frac{L_0}2 - \frac{L_1}4 + \frac{L_2}8 - \frac{L_3}{16} + \frac{E_4}{16},
  \end{align}
  where
  \begin{align}
    L_0 &= \bE_Z[Z_+-Z_-],\\
    L_1 &= \bE_Z[(Z_+-Z_-)(Z_++Z_--2)],\\
    L_2 &= \bE_Z[(Z_+-Z_-)(Z_++Z_--2)^2],\\
    L_3 &= \bE_Z[(Z_+-Z_-)(Z_++Z_--2)^3],\\
    E_4 &= \bE_Z\left[\frac{(Z_+-Z_-)(Z_++Z_--2)^4}{Z_++Z_-}\right]
    \le \bE_Z (Z_++Z_--2)^4.
  \end{align}
  Computing the terms using
  \begin{align}
    \bE_Z \exp(\mu + \sigma Z) = e^{\mu+\frac{\sigma^2}2},
  \end{align}
  we get
  \begin{align}
    L_0 &= e^{3s/2}-e^{-s/2}, \\
    L_1 &= e^{4s}-2e^{3s/2}+2e^{-s/2}-1, \\
    L_2 &= e^{15s/2}-4e^{4s}+4e^{3s/2}-5e^{-s/2}+4, \\
    L_3 &= e^{12s}-6e^{15s/2}+13e^{4s}-8e^{3s/2}+14e^{-s/2}-14, \\
    E_4 &\le e^{12s}-8e^{15s/2}+29e^{4s}-64e^{3s/2}-56e^{-s/2}+98.
  \end{align}
  Collecting terms, we get
  \begin{align} \label{eqn:large-step:g6-bound-step}
    g(s) \le \frac 14 \left(e^{4s}-8e^{3s/2}-24e^{-s/2}+31\right).
  \end{align}

  When $x\le 0.1$, we have $s\le 0.105$. So the largest exponent is $0.105 \cdot 4 \le 0.5$.
  By Taylor's theorem, for $y\le 0.5$, we have
  \begin{align}
    \left| e^y-\sum_{0\le i\le 3} \frac{y^i}{i!}\right| \le e^{0.5} \frac{y^4}{4!} \le \frac{y^4}{12}.
  \end{align}
  Applying this to Eq.~\eqref{eqn:large-step:g6-bound-step}, we get
  \begin{align}
    g(s) \le s-s^2+\frac 53 s^3+\frac{135}{16}s^4 \le s-\frac 23 s^2,
  \end{align}
  where the second step uses $s\le 0.105$.
  Note that $s-\frac 23 s^2$ is increasing for $s\in [0, 0.105]$.
  Because $x\le 0.1$, we have
  \begin{align}
    s=x+2x^3+\frac{x^5}5 \le x+0.2002 x^2 \le x+\frac{x^2}4.
  \end{align}
  Therefore
  \begin{align}
    g(s) \le \left(x+\frac{x^2}4\right)-\frac 23 \left(x+\frac{x^2}4\right)^2 = x-\frac{5x^2}{12}-\frac{x^3}3-\frac{x^4}{24} < x.
  \end{align}

  \textbf{Case $x\in [\delta, 1]$.}
  Because $\frac{ds}{dx}=1+6x^2+x^4$, $s$ is $8$-Lipschitz for $x\in [\delta, 1]$.
  For $x\ge \delta$, $\frac{d\sqrt s}{dx} = \frac {1+6x^2+x^4}{2\sqrt{x+2x^3+\frac{x^5}5}} \le 2.5$. So $\sqrt s$ is $2.5$-Lipschitz for $x\in [\delta, 1]$.

  Let $x,x'\in [\delta, 1]$, $s=x+2x^3+\frac{x^5}5$, $s=x'+2x^{\prime 3}+\frac{x^{\prime 5}}5$.
  Then
  \begin{align}
    |g(s)-g(s')| &\le |s-s'| + \bE_{Z\sim \cN(0,1)} |Z| \left|\sqrt s-\sqrt{s'}\right| \\
    \nonumber &= |s-s'| + \sqrt{\frac 2\pi} \left|\sqrt s-\sqrt{s'}\right| \\
    \nonumber &\le  10 |x-x'|.
  \end{align}
  Consider the set $S = \{\delta, \delta+0.001, \ldots, 1\}$.
  Suppose we know that $g\left(x+2x^{3}+\frac{x^{5}}5\right) \le x-0.006$ for all $x\in S$.
  Then for any $x\in [\delta, 1]$, there exists $x^*\in S$ such that $|x-x^*|\le 0.0005$, and
  \begin{align}
    g\left(x+2x^{3}+\frac{x^{5}}5\right) \le g\left(x^*+2x^{*3}+\frac{x^{*5}}5\right) + 0.005
    \le x^*-0.001 \le x-0.0005.
  \end{align}
  Therefore it remains to prove that $g\left(x+2x^{3}+\frac{x^{5}}5\right) \le x-0.006$ for all $x\in S$.
  We do this using a rigorous numerical integration.
  Suppose we would like to evaluate $g(s)$ where $s=x+2x^{3}+\frac{x^{5}}5$.
  Let $n$ be a positive integer to determined later.
  \begin{align}
    g(s) &= \bE_{Z\sim \cN(0,1)} \tanh(s+\sqrt s Z)\\
    \nonumber &= \bP[|Z|\ge 5] + \int_{-5}^5 \tanh(s+\sqrt s t) \frac{e^{-\frac{t^2}2}}{\sqrt{2\pi}} dt \\
    \nonumber &\le 10^{-6} + \int_{-5}^5 \tanh(s+\sqrt s t) \frac{e^{-\frac{t^2}2}}{\sqrt{2\pi}} dt \\
    \nonumber &\le 10^{-6} + \sum_{0\le i\le 10n-1} \frac 1n \tanh\left(s+\sqrt s\left(-5+\frac{i+1}n\right)\right) \frac{e^{-\frac 12 \max\{(-5+\frac in)^2,(-5+\frac{i+1}n)^2\}}}{\sqrt{2\pi}}.
  \end{align}
  The last expression can be evaluated to arbitrarily high precision using e.g., Mathematica.
  Computation shows that taking $n=1000$ suffices.
\end{proof}

\subsection{A priori estimates} \label{sec:boht-app:a-priori}
The goal of this section is to prove Lemma~\ref{lem:boht-robust:a-priori} and its corollary Lemma~\ref{lem:boht-robust:normal}.

\begin{proof}[Proof of Lemma~\ref{lem:boht-robust:a-priori}, Eq.~\eqref{eqn:lem-boht-robust-a-priori:i}]
  Given $x_k = C_{\chi^2}(M_{P,k})$, we would like to compute $x'_{k+1} = C_{\chi^2}(M'_{P,k+1})$ where $M'_{P,k+1} = M_{P,k}^{\times (r-1)} \circ B_{r,\lambda}$.
  Let $\theta_1,\ldots,\theta_{r-1}$ be i.i.d.~copies of $\theta_{M_{P,k}}$. Then
  \begin{align}
    x'_{k+1} &= C_{\chi^2}\left(M_{P,k}^{\times (r-1)} \circ B_{r,\lambda}\right) \\
    \nonumber &= \bE_{\theta_1,\ldots,\theta_{r-1}} C_{\chi^2}((B_{\theta_1}\times \cdots \times B_{\theta_{r-1}})\circ B_{r,\lambda}),
  \end{align}
  where $B_\theta$ denotes the channel $\BSC_{\frac{1-\theta}2}$.

  For fixed $\theta_1,\ldots,\theta_{r-1}\in [-1,1]$, write $R = (B_{\theta_1} \times \cdots \times B_{\theta_{r-1}}) \circ B_{r,\lambda} : \{\pm\} \to \{\pm\}^{r-1}$.
  For $y\in \{\pm\}^{r-1}$, we have
  \begin{align}
    R_+(y) &= \sum_{z\in \{\pm\}^{r-1}} (B_{r,\lambda})_+(z) \prod_{i\in [r-1]} B_{\theta_i}(y_i|z_i) \\
    \nonumber &= \sum_{z\in \{\pm\}^{r-1}} \left( \frac {1-\lambda}{2^{r-1}} + \lambda \mathbbm{1}\{z=+^{r-1}\} \right) \prod_{i\in [r-1]} \frac 12 (1+\theta_i y_i z_i) \\
    \nonumber &= \frac {1-\lambda}{2^{r-1}} + \frac{\lambda}{2^{r-1}} \prod_{i\in [r-1]} (1+\theta_i y_i).
  \end{align}
  For simplicity, for two sequences $\theta\in [-1,1]^{r-1}$, $y\in \{\pm\}^{r-1}$, define
  \begin{align}
    \alpha(\theta,y) = \prod_{i\in [r-1]} (1+\theta_i y_i).
  \end{align}
  Then we have
  \begin{align} \label{eqn:boht-app:a-priori:i-R-plus}
    R_+(y) = \frac {1-\lambda}{2^{r-1}} + \frac{\lambda}{2^{r-1}} \alpha(\theta, y).
  \end{align}
  Then $\theta_R(y)$, the $\theta$-component corresponding to $y\in \{\pm\}^{r-1}$, is described as
  \begin{align} \label{eqn:boht-app:a-priori:i-theta-R}
    \theta_R(y) = \frac{R_+(y)-R_-(y)}{R_+(y)+R_-(y)}
    = \frac{\lambda (\alpha(\theta,y) - \alpha(\theta,-y))}{2(1-\lambda) + \lambda (\alpha(\theta,y) + \alpha(\theta,-y))}.
  \end{align}
  Combining Eq.~\eqref{eqn:boht-app:a-priori:i-R-plus} and Eq.~\eqref{eqn:boht-app:a-priori:i-theta-R}, we get
  \begin{align} \label{eqn:boht-app:a-priori:i-chi2-cap-R}
    C_{\chi^2}(R) = \sum_{y\in \{\pm\}^{r-1}} R_+(y) \theta_R(y)
    = \frac{\lambda^2}{2^r} \sum_{y\in \{\pm\}^{r-1}} \frac{\alpha(\theta,y) (\alpha(\theta,y)-\alpha(\theta,-y))}{1+\frac \lambda 2 (\alpha(\theta,y) + \alpha(\theta,-y) - 2)}.
  \end{align}
  Expanding the last term using
  \begin{align} \label{eqn:boht-app:a-priori:i-expand}
    \frac 1{1+x} = 1-x+\frac{x^2}{1+x},
  \end{align}
  we get
  \begin{align}
    C_{\chi^2}(R) = \frac{\lambda^2}{2^r} \left( T_0(\theta) - \frac \lambda2 T_1(\theta) +\frac{\lambda^2}4 R_2(\theta) \right),
  \end{align}
  where
  \begin{align}
    \label{eqn:boht-app:a-priori:i-T0} T_0(\theta) &= \sum_{y\in \{\pm\}^{r-1}} \alpha(\theta,y)(\alpha(\theta,y)-\alpha(\theta,-y)), \\
    \label{eqn:boht-app:a-priori:i-T1} T_1(\theta) &= \sum_{y\in \{\pm\}^{r-1}} \alpha(\theta,y)(\alpha(\theta,y)-\alpha(\theta,-y))(\alpha(\theta,y)+\alpha(\theta,-y)-2),\\
    R_2(\theta) &= \sum_{y\in \{\pm\}^{r-1}} \frac{\alpha(\theta,y)(\alpha(\theta,y)-\alpha(\theta,-y))(\alpha(\theta,y)+\alpha(\theta,-y)-2)^2}{1+\frac \lambda 2 (\alpha(\theta,y)+\alpha(\theta,-y)-2)}.
  \end{align}
  Note that
  \begin{align}
    \frac \lambda 2 |R_2(\theta)| &\le \sum_{y\in \{\pm\}^{r-1}} \alpha(\theta,y)(\alpha(\theta,y)+\alpha(\theta,-y)-2)^2 \left|\frac{R_+(y)-R_-(y)}{R_+(y)+R_-(y)}\right| \\
    \nonumber &\le \sum_{y\in \{\pm\}^{r-1}} \alpha(\theta,y)(\alpha(\theta,y)+\alpha(\theta,-y)-2)^2.
  \end{align}

  By Lemma~\ref{lem:boht-app:a-priori:moment}, we have
  \begin{align}
    \label{eqn:boht-app:a-priori:i-T0-comp} \bE T_0(\theta) &= 2^{r-1}\left((1+x_k)^{r-1} -(1-x_k)^{r-1}\right), \\
    \label{eqn:boht-app:a-priori:i-T1-comp} \bE T_1(\theta) &= 2^{r-1}\left((1+3 x_k)^{r-1} -2 (1+x_k)^{r-1} + (1-x_k)^{r-1}\right),\\
    \frac \lambda 2 \bE |R_2(\theta)| &\le 2^{r-1}\left((1+3 x_k)^{r-1}-(1-x_k)^{r-1}-4(1+x_k)^{r-1}+4\right).
  \end{align}
  Expanding with respect to $x_k$, we get
  \begin{align}
    \bE T_0(\theta) &= 2^r (r-1) x_k + O_r(x_k^2),\\
    \bE T_1(\theta) &= O_r(x_k^2),\\
    \frac \lambda 2 \bE |R_2(\theta)| & = O_r(x_k^2).
  \end{align}

  Finally,
  \begin{align}
    x'_{k+1} &= \bE_{\theta_1,\ldots,\theta_{r-1}} C_{\chi^2}((B_{\theta_1}\times \cdots \times B_{\theta_{r-1}})\circ B_{r,\lambda}) \\
    \nonumber &= \bE_{\theta_1,\ldots,\theta_{r-1}} \frac{\lambda^2}{2^r} \left( T_0(\theta) - \frac \lambda2 T_1(\theta) +\frac{\lambda^2}4 R_2(\theta) \right) \\
    \nonumber &= (r-1) \lambda^2 x_k + O_r(\lambda^2 x_k^2).
  \end{align}
  This finishes the proof.
\end{proof}

\begin{proof}[Proof of Lemma~\ref{lem:boht-robust:a-priori}, Eq.~\eqref{eqn:lem-boht-robust-a-priori:ii}]
  Use the same notation as in the proof of Lemma~\ref{lem:boht-robust:a-priori}, Eq.~\eqref{eqn:lem-boht-robust-a-priori:i}.
  We need to compute
  \begin{align}
    \bE \theta_R^3 = \sum_{y\in \{\pm\}^{r-1}} R_+(y) \theta_R^3(y)
    = \frac{\lambda^4}{2^{r+2}} \sum_{y\in \{\pm\}^{r-1}} \frac{\alpha(\theta,y)(\alpha(\theta,y)-\alpha(\theta,-y))^3}{\left(1+\frac \lambda 2 (\alpha(\theta,y)+\alpha(\theta,-y)-2)\right)^3}.
  \end{align}
  Expanding using
  \begin{align} \label{eqn:boht-app:a-priori:ii-expand}
    \frac 1{(1+x)^3} = 1-3x + \frac{6x^2+8x^3+3x^4}{(1+x)^3},
  \end{align}
  we get
  \begin{align}
    \bE \theta_R^3 = \sum_{y\in \{\pm\}^{r-1}} R_+(y) \theta_R^3(y) = \frac{\lambda^4}{2^{r+2}} \left(S_0(\theta) - \frac{3\lambda}{2} S_1(\theta) + E_2(\theta)\right),
  \end{align}
  where
  \begin{align}
    S_0(\theta) &= \sum_{y\in \{\pm\}^{r-1}} \alpha(\theta,y)(\alpha(\theta,y)-\alpha(\theta,-y))^3,\\
    S_1(\theta) &= \sum_{y\in \{\pm\}^{r-1}} \alpha(\theta,y)(\alpha(\theta,y)-\alpha(\theta,-y))^3(\alpha(\theta,y)+\alpha(\theta,-y)-2),\\
    E_2(\theta) &= \sum_{y\in \{\pm\}^{r-1}} \alpha(\theta,y)(\alpha(\theta,y)-\alpha(\theta,-y))^3 \cdot \frac{6x^2+8x^3+3x^4}{(1+x)^3}
  \end{align}
  where $x = \frac \lambda2 (\alpha(\theta,y)+\alpha(\theta,-y)-2)$.

  Note that
  \begin{align}
    \frac{\lambda^3}{8} |E_2(\theta)| &\le \sum_{y\in \{\pm\}^{r-1}} \alpha(\theta,y) (6x^2+8x^3+3x^4) \left|\frac{R_+(y)-R_-(y)}{R_+(y)+R_-(y)}\right|^3 \\
    \nonumber &\le \sum_{y\in \{\pm\}^{r-1}} \alpha(\theta,y) (6x^2+8x^3+3x^4).
  \end{align}

  By Lemma~\ref{lem:boht-app:a-priori:moment}, we have
  \begin{align}
    \bE S_0(\theta) &= 2^{r-1} \left((1+6x_k+y_k)^{r-1}-4(1-y_k)^{r-1}+3(1-2x_k+y_k)^{r-1}\right),\\
    \bE S_1(\theta) &= 2^{r-1} \left((1+10x_k+5y_k)^{r-1}-3(1+2x_k-3y_k)^{r-1}-4(1-2x_k+y_k)^{r-1}\right. \\
    \nonumber &\left.-2(1+6x_k+y_k)^{r-1}+8(1-y_k)^{r-1}\right),\\
    \frac{\lambda^3}{8} \bE |E_2(\theta)| &\le \frac {3 \lambda^2}2 \cdot 2^{r-1} \left((1+3 x_k)^{r-1}-(1-x_k)^{r-1}-4(1+x_k)^{r-1}+4\right) \\
    \nonumber &+ \lambda^3 \cdot 2^{r-1} \left(
      (1+6x_k+y_k)^{r-1}+4(1-y_k)^{r-1}-6(1+3x_k)^{r-1}\right.\\
    \nonumber &\left.+3(1-2x_k+y_k)^{r-1}-6(1-x_k)^{r-1}+12(1+x_k)^{r-1}-8
    \right) \\
    \nonumber &+ \frac{3\lambda^4}{16} \cdot 2^{r-1} \left(
      (1+10x_k+5y_k)^{r-1}+5(1+2x_k-3y_k)^{r-1}-8(1+6x_k+y_k)^{r-1}\right.\\
    \nonumber & +10(1-2x_k+y_k)^{r-1}-32(1-y_k)^{r-1}+24(1+3x_k)^{r-1}-24(1-2x_k+y_k)^{r-1}\\
    \nonumber &\left.+40(1-x_k)^{r-1}-32(1+x_k)^{r-1}+16\right).
  \end{align}

  Expanding with respect to $x_k$ (and using that $y_k\le x_k$) we get
  \begin{align}
    \bE S_0(\theta) &= 2^{r+2} (r-1) y_k + O_r(x_k^2),\\
    \bE S_1(\theta) &= O_r(x_k^2),\\
    \frac{\lambda^3}{8} \bE |E_2(\theta)| &= O_r(\lambda^2 x_k^2).
  \end{align}

  Finally,
  \begin{align}
    y'_{k+1} &= \bE_{\theta_1,\ldots,\theta_{r-1}} \bE \theta_R^3 \\
    \nonumber & = \bE_{\theta_1,\ldots,\theta_{r-1}} \frac{\lambda^4}{2^{r+2}} \left(S_0(\theta) - \frac{3\lambda}{2} S_1(\theta) + E_2(\theta)\right) \\
    \nonumber &= (r-1) \lambda^4 y_k + O_r(\lambda^2 x_k^2).
  \end{align}
  This finishes the proof.
\end{proof}

\begin{proof}[Proof of Lemma~\ref{lem:boht-robust:a-priori}, Eq.~\eqref{eqn:lem-boht-robust-a-priori:iii}]
  Given $x'_{k+1}=C_{\chi^2}(M'_{P,k+1})$, we would like to compute $x_{k+1} = C_{\chi^2}(M_{P,k+1})$, where $M_{P,k+1}=\bE_{b\sim D} (M'_{P,k+1})^{\star b}$.
  Let $(\theta_i)_{i\in \bZ_{\ge 0}}$ be i.i.d.~copies of $\theta_{M'_{P,k+1}}$.
  Then
  \begin{align}
    x_{k+1} &= C_{\chi^2}\left(\bE_{b\sim D} (M'_{P,k+1})^{\star b}\right) \\
    &= \bE_{b\sim D} \bE_\theta C_{\chi^2}(B_{\theta_1}\star \cdots \star B_{\theta_b}).
  \end{align}

  For fixed $b$ and $(\theta_i)_{i\in \bZ_{\ge 0}}$, write $R=B_{\theta_1}\star \cdots \star B_{\theta_b}: \{\pm\} \to \{\pm\}^b$.
  For $y\in \{\pm\}^b$, we have
  \begin{align} \label{eqn:boht-app:a-priori:iii-R-plus}
    R_+(y) = \prod_{i\in [b]} (B_{\theta_i})_+(y_i) = \frac 1{2^b} \alpha(\theta, y).
  \end{align}
  Then $\theta_R(y)$, the $\theta$-component corresponding to $y\in \{\pm\}^b$, is described as
  \begin{align} \label{eqn:boht-app:a-priori:iii-theta-R}
    \theta_R(y) = \frac{R_+(y)-R_-(y)}{R_+(y)+R_-(y)} = \frac{\alpha(\theta,y)-\alpha(\theta,-y)}{\alpha(\theta,y)+\alpha(\theta,-y)}.
  \end{align}
  Combining Eq.~\eqref{eqn:boht-app:a-priori:iii-R-plus} and Eq.~\eqref{eqn:boht-app:a-priori:iii-theta-R}, we get
  \begin{align} \label{eqn:boht-app:a-priori:iii-chi2-cap-R}
    C_{\chi^2}(R) = \sum_{y\in \{\pm\}^b} R_+(y) \theta_R(y) = \frac 1{2^{b+1}} \sum_{y\in \{\pm\}^b} \frac{\alpha(\theta,y)-\alpha(\theta,-y)}{1+\frac 12(\alpha(\theta,y)+\alpha(\theta,-y)-2)}.
  \end{align}
  Expanding the last term using Eq.~\eqref{eqn:boht-app:a-priori:i-expand}, we get
  \begin{align}
    C_{\chi^2}(R) = \frac{1}{2^{b+1}} \left( T_0(\theta) - \frac 12 T_1(\theta) +\frac{1}4 R_2(\theta) \right),
  \end{align}
  where
  \begin{align}
    \label{eqn:boht-app:a-priori:iii-T0} T_0(\theta) &= \sum_{y\in \{\pm\}^b} \alpha(\theta,y)(\alpha(\theta,y)-\alpha(\theta,-y)), \\
    \label{eqn:boht-app:a-priori:iii-T1} T_1(\theta) &= \sum_{y\in \{\pm\}^b} \alpha(\theta,y)(\alpha(\theta,y)-\alpha(\theta,-y))(\alpha(\theta,y)+\alpha(\theta,-y)-2),\\
    R_2(\theta) &= \sum_{y\in \{\pm\}^b} \frac{\alpha(\theta,y)(\alpha(\theta,y)-\alpha(\theta,-y))(\alpha(\theta,y)+\alpha(\theta,-y)-2)^2}{1+\frac 12 (\alpha(\theta,y)+\alpha(\theta,-y)-2)}.
  \end{align}

  Using the same computation as the proof of Lemma~\ref{lem:boht-robust:a-priori}, Eq.~\eqref{eqn:lem-boht-robust-a-priori:i}, we get
  \begin{align}
    \bE_\theta T_0(\theta) &= 2^b\left((1+x'_{k+1})^b -(1-x'_{k+1})^b\right), \\
    \bE_\theta T_1(\theta) &= 2^b\left((1+3 x'_{k+1})^b -2 (1+x'_{k+1})^b + (1-x'_{k+1})^b\right),\\
    \frac 12 \bE_\theta R_2(\theta) &\le 2^b\left((1+3 x'_{k+1})^b-(1-x'_{k+1})^b-4(1+x'_{k+1})^b+4\right).
  \end{align}
  By Eq.~\eqref{eqn:lem-boht-robust-a-priori:i}, $x'_{k+1}=O_r(\lambda^2 x_k)$.
  Using Lemma~\ref{lem:boht-robust:low-deg-poly}, we have
  \begin{align}
    \bE_b \frac 1{2^{b+1}} \bE_\theta T_0(\theta) &= d x'_{k+1} + O_r(d^2 x_{k+1}^{\prime 2}),\\
    \bE_b \frac 1{2^{b+1}} \bE_\theta T_1(\theta) &= O_r(d^2 x_{k+1}^{\prime 2}),\\
    \bE_b \frac 1{2^{b+1}} \bE_\theta R_2(\theta) & = O_r(d^2 x_{k+1}^{\prime 2}).
  \end{align}

  Finally,
  \begin{align}
    x_{k+1} &= \bE_b \bE_{\theta_1,\ldots,\theta_b} C_{\chi^2}(B_{\theta_1}\star \cdots \star B_{\theta_b}) \\
    \nonumber &= \bE_b \bE_{\theta_1,\ldots,\theta_b} \frac{1}{2^{b+1}} \left( T_0(\theta) - \frac 12 T_1(\theta) +\frac{1}4 R_2(\theta) \right) \\
    \nonumber &= d x'_{k+1} + O_r(d^2 x_{k+1}^{\prime 2}).
  \end{align}
  This finishes the proof.
\end{proof}

\begin{proof}[Proof of Lemma~\ref{lem:boht-robust:a-priori}, Eq.~\eqref{eqn:lem-boht-robust-a-priori:iv}]
  Use the same notation as in the proof of Lemma~\ref{lem:boht-robust:a-priori}, Eq.~\eqref{eqn:lem-boht-robust-a-priori:iii}.
  We need to compute
  \begin{align}
    \bE \theta_R^3 = \sum_{y\in \{\pm\}^b} R_+(y) \theta_R^3(y)
    = \frac{1}{2^{b+3}} \sum_{y\in \{\pm\}^b} \frac{\alpha(\theta,y)(\alpha(\theta,y)-\alpha(\theta,-y))^3}{\left(1+\frac 1 2 (\alpha(\theta,y)+\alpha(\theta,-y)-2)\right)^3}.
  \end{align}
  Expanding using Eq.~\eqref{eqn:boht-app:a-priori:ii-expand}, we get
  \begin{align}
    \bE \theta_R^3 = \sum_{y\in \{\pm\}^b} R_+(y) \theta_R^3(y) = \frac{1}{2^{b+3}} \left(S_0(\theta) - \frac{3}{2} S_1(\theta) + E_2(\theta)\right),
  \end{align}
  where
  \begin{align}
    S_0(\theta) &= \sum_{y\in \{\pm\}^b} \alpha(\theta,y)(\alpha(\theta,y)-\alpha(\theta,-y))^3,\\
    S_1(\theta) &= \sum_{y\in \{\pm\}^b} \alpha(\theta,y)(\alpha(\theta,y)-\alpha(\theta,-y))^3(\alpha(\theta,y)+\alpha(\theta,-y)-2),\\
    E_2(\theta) &= \sum_{y\in \{\pm\}^b} \alpha(\theta,y)(\alpha(\theta,y)-\alpha(\theta,-y))^3 \cdot \frac{6x^2+8x^3+3x^4}{(1+x)^3}
  \end{align}
  where $x = \frac 12 (\alpha(\theta,y)+\alpha(\theta,-y)-2)$.

  Using the same computation as the proof of Lemma~\ref{lem:boht-robust:a-priori}, Eq.~\eqref{eqn:lem-boht-robust-a-priori:ii}, we get
  \begin{align}
    \bE_\theta S_0(\theta) &= 2^b \left((1+6x'_{k+1}+y'_{k+1})^b-4(1-y'_{k+1})^b+3(1-2x'_{k+1}+y'_{k+1})^b\right),\\
    \bE_\theta S_1(\theta) &= 2^b \left((1+10x'_{k+1}+5y'_{k+1})^b-3(1+2x'_{k+1}-3y'_{k+1})^b\right. \\
    \nonumber &\left.-4(1-2x'_{k+1}+y'_{k+1})^b-2(1+6x'_{k+1}+y'_{k+1})^b+8(1-y'_{k+1})^b\right),\\
    \frac{1}{8} \bE_\theta |E_2(\theta)| &\le \frac {3}2 \cdot 2^b \left((1+3 x'_{k+1})^b-(1-x'_{k+1})^b-4(1+x'_{k+1})^b+4\right) \\
    \nonumber &+ 2^b \left(
      (1+6x'_{k+1}+y'_{k+1})^b+4(1-y'_{k+1})^b-6(1+3x'_{k+1})^b\right.\\
    \nonumber &\left.+3(1-2x'_{k+1}+y'_{k+1})^b-6(1-x'_{k+1})^b+12(1+x'_{k+1})^b-8
    \right) \\
    \nonumber &+ \frac{3}{16} \cdot 2^b \left(
      (1+10x'_{k+1}+5y'_{k+1})^b+5(1+2x'_{k+1}-3y'_{k+1})^b\right.\\
    \nonumber &-8(1+6x'_{k+1}+y'_{k+1})^b  +10(1-2x'_{k+1}+y'_{k+1})^b-32(1-y'_{k+1})^b\\
    \nonumber &+24(1+3x'_{k+1})^b-24(1-2x'_{k+1}+y'_{k+1})^b+40(1-x'_{k+1})^b\\
    \nonumber &\left.-32(1+x'_{k+1})^b+16\right).
  \end{align}

  Using $y'_{k+1}\le x'_{k+1}=O_r(\lambda^2 x_k)$ and Lemma~\ref{lem:boht-robust:low-deg-poly}, we have
  \begin{align}
    \bE_b \frac 1{2^{b+3}} \bE_\theta S_0(\theta) &= d y'_{k+1} + O_r(d^2 x_{k+1}^{\prime 2}),\\
    \bE_b \frac 1{2^{b+3}} \bE_\theta S_1(\theta) &= O_r(d^2 x_{k+1}^{\prime 2}),\\
    \bE_b \frac 1{2^{b+3}} \bE_\theta |E_2(\theta)| &= O_r(d^2 x_{k+1}^{\prime 2}).
  \end{align}

  Finally,
  \begin{align}
    y_{k+1} &= \bE_b \bE_{\theta_1,\ldots,\theta_b} \bE \theta_R^3 \\
    \nonumber & = \bE_b \bE_{\theta_1,\ldots,\theta_b} \frac{1}{2^{b+3}} \left(S_0(\theta) - \frac{3}{2} S_1(\theta) + E_2(\theta)\right) \\
    \nonumber &= d y'_{k+1} + O_r(d^2 x_{k+1}^{\prime 2}).
  \end{align}
  This finishes the proof.
\end{proof}

\begin{lemma}[Moment calculations] \label{lem:boht-app:a-priori:moment}
  Let $P$ be a BMS channel and $\ell\in \bZ_{\ge 0}$.
  Let $x = \bE \theta_P$ and $z = \bE \theta_P^3$.
  Let $\theta_1,\ldots,\theta_\ell$ be i.i.d.~copies of $\theta_P$, the $\theta$-component of $P$.
  For $y\in \{\pm\}^\ell$, define $\alpha(\theta,y) = \prod_{i\in [\ell]} (1+\theta_i y_i)$.
  Then we have
  \begin{align}
    \bE \sum_{y\in \{\pm\}^\ell} \alpha(\theta,y) &= 2^\ell,\\
    \bE \sum_{y\in \{\pm\}^\ell} \alpha^2(\theta,y) &= 2^\ell (1+x)^\ell,\\
    \bE \sum_{y\in \{\pm\}^\ell} \alpha(\theta,y)\alpha(\theta,-y) &= 2^\ell (1-x)^\ell, \\
    \bE \sum_{y\in \{\pm\}^\ell} \alpha^3(\theta,y) &= 2^\ell (1+3x)^\ell,\\
    \bE \sum_{y\in \{\pm\}^\ell} \alpha^2(\theta,y)\alpha(\theta,-y) &= 2^\ell (1-x)^\ell, \\
    \bE \sum_{y\in \{\pm\}^\ell} \alpha^4(\theta,y) &= 2^\ell (1+6x+z)^\ell,\\
    \bE \sum_{y\in \{\pm\}^\ell} \alpha^3(\theta,y)\alpha(\theta,-y) &= 2^\ell (1-z)^\ell,\\
    \bE \sum_{y\in \{\pm\}^\ell} \alpha^2(\theta,y)\alpha^2(\theta,-y) &= 2^\ell (1-2x+z)^\ell,\\
    \bE \sum_{y\in \{\pm\}^\ell} \alpha^5(\theta,y) &= 2^\ell (1+10x+5z)^\ell,\\
    \bE \sum_{y\in \{\pm\}^\ell} \alpha^4(\theta,y)\alpha(\theta,-y) &= 2^\ell (1+2x-3z)^\ell,\\
    \bE \sum_{y\in \{\pm\}^\ell} \alpha^3(\theta,y)\alpha^2(\theta,-y) &= 2^\ell (1-2x+z)^\ell.
  \end{align}
\end{lemma}
\begin{proof}
  The proof is by direct computations.
  \paragraph{First moment}
  \begin{align}
    \sum_{y\in \{\pm\}^\ell} \alpha(\theta,y) &= \sum_{y\in \{\pm\}^\ell} \prod_{i\in [\ell]} (1+\theta_i y_i) = 2^\ell.
  \end{align}
  \paragraph{Second moments}
  \begin{align}
    \sum_{y\in \{\pm\}^\ell} \alpha^2(\theta,y) &= \sum_{y\in \{\pm\}^\ell} \prod_{i\in [\ell]} (1 + 2 \theta_i y_i + \theta_i^2) = 2^\ell \prod_{i\in [\ell]} (1+\theta_i^2), \\
    \sum_{y\in \{\pm\}^\ell} \alpha(\theta,y)\alpha(\theta,-y) &= \sum_{y\in \{\pm\}^\ell} \prod_{i\in [\ell]} (1-\theta_i^2) = 2^\ell \prod_{i\in [\ell]} (1-\theta_i^2).
  \end{align}
  \paragraph{Third moments}
  \begin{align}
    \sum_{y\in \{\pm\}^\ell} \alpha^3(\theta,y) &= \sum_{y\in \{\pm\}^\ell} \prod_{i\in [\ell]} (1 + 3 \theta_i y_i + 3 \theta_i^2 + \theta_i^3 y_i) = 2^\ell \prod_{i\in [\ell]} (1+3 \theta_i^2), \\
    \sum_{y\in \{\pm\}^\ell} \alpha^2(\theta,y)\alpha(\theta,-y) &= \sum_{y\in \{\pm\}^\ell} \prod_{i\in [\ell]} (1 + \theta_i y_i - \theta_i^2 - \theta_i^3 y_i) = 2^\ell \prod_{i\in [\ell]} (1-\theta_i^2).
  \end{align}
  \paragraph{Fourth moments}
  \begin{align}
    \sum_{y\in \{\pm\}^\ell} \alpha^4(\theta,y) &= \sum_{y\in \{\pm\}^\ell} \prod_{i\in [\ell]} (1 + 4 \theta_i y_i + 6 \theta_i^2 + 4 \theta_i^3 y_i + \theta_i^4) \\
    \nonumber &= 2^\ell \prod_{i\in [\ell]} (1 + 6 \theta_i^2 + \theta_i^4),\\
    \sum_{y\in \{\pm\}^\ell} \alpha^3(\theta,y)\alpha(\theta,-y) &= \sum_{y\in \{\pm\}^\ell} \prod_{i\in [\ell]} (1 + 2 \theta_i y_i -2 \theta_i^3 y_i - \theta_i^4) = 2^\ell \prod_{i\in [\ell]} (1 - \theta_i^4), \\
    \sum_{y\in \{\pm\}^\ell} \alpha^2(\theta,y)\alpha^2(\theta,-y) &= \sum_{y\in \{\pm\}^\ell} \prod_{i\in [\ell]} (1 - 2 \theta_i^2 + \theta_i^4) = 2^\ell \prod_{i\in [\ell]} (1 - 2 \theta_i^2 + \theta_i^4).
  \end{align}
  \paragraph{Fifth moments}
  \begin{align}
    \sum_{y\in \{\pm\}^\ell} \alpha^5(\theta,y) &= \sum_{y\in \{\pm\}^\ell} \prod_{i\in [\ell]} (1 + 5 \theta_i y_i + 10 \theta_i^2 + 10 \theta_i^3 y_i + 5 \theta_i^4 + \theta_i^5) \\
    \nonumber &= 2^\ell \prod_{i\in [\ell]} (1 + 10 \theta_i^2 + 5 \theta_i^4),\\
    \sum_{y\in \{\pm\}^\ell} \alpha^4(\theta,y)\alpha(\theta,-y) &= \sum_{y\in \{\pm\}^\ell} \prod_{i\in [\ell]} (1 + 3 \theta_i y_i + 2 \theta_i^2 -2 \theta_i^3 y_i - 3 \theta_i^4 - \theta_i^5) \\
    \nonumber &= 2^\ell \prod_{i\in [\ell]} (1 + 2 \theta_i^2 - 3 \theta_i^4), \\
    \sum_{y\in \{\pm\}^\ell} \alpha^3(\theta,y)\alpha^2(\theta,-y) &= \sum_{y\in \{\pm\}^\ell} \prod_{i\in [\ell]} (1 + \theta_i y_i - 2 \theta_i^2 - 2 \theta_i^3 y_i + \theta_i^4 + \theta_i^5) \\
    \nonumber &= 2^\ell \prod_{i\in [\ell]} (1 - 2 \theta_i^2 + \theta_i^4).
  \end{align}
\end{proof}

\begin{lemma} \label{lem:boht-robust:low-deg-poly}
  There exists $\epsilon>0$ and $C>0$ such that for all $d\ge 0$ and $\lambda\in [-1,1]$ satisfying $(r-1) d \lambda^2 \le 1$, for all $x \le \lambda^2 \epsilon$ and $m\in \bZ_{\ge 0}$, we have
  \begin{align}
    \left| \bE_{b\sim D} (1+x)^b - \sum_{0\le i\le m} \bE_{b\sim D} \binom b i x^i \right| \le C (d x)^{m+1}.
  \end{align}
  Here $D$ denotes either the point distribution at $d$, or $\Pois(d)$.
\end{lemma}
\begin{proof}
  Follows from the proof of \cite[Lemma 3.7]{mossel2023exact}.
\end{proof}

Lemma~\ref{lem:boht-robust:low-deg-poly} is the only place in the proof of Theorem~\ref{thm:boht-robust} where we use that $D$ is a point distribution or a Poisson distribution.
We remark that \cite[Lemma 3.7]{mossel2023exact} works for a large class of offspring distributions $D$ (see \cite[Assumption 1.10 and 1.11]{mossel2023exact}). Therefore Theorem~\ref{thm:boht-robust} and Theorem~\ref{thm:boht-r56} holds for more general offspring distributions by replacing Lemma~\ref{lem:boht-robust:low-deg-poly} with a more general version.

Lemma~\ref{lem:boht-robust:a-priori}, Eq.~\eqref{eqn:lem-boht-robust-a-priori:final-i} and Eq.~\eqref{eqn:lem-boht-robust-a-priori:final-ii} directly follow from Eq.~\eqref{eqn:lem-boht-robust-a-priori:i}, Eq.~\eqref{eqn:lem-boht-robust-a-priori:ii}, Eq.~\eqref{eqn:lem-boht-robust-a-priori:iii}, Eq.~\eqref{eqn:lem-boht-robust-a-priori:iv}.
Thus we have finished the proof of Lemma~\ref{lem:boht-robust:a-priori}.

\begin{proof}[Proof of Lemma~\ref{lem:boht-robust:normal}]
  Because the statement is monotone in $\delta$, WLOG assume that $C_{\chi^2}(P)=\delta$.
  By Lemma~\ref{lem:boht-robust:a-priori}, there exists $C=C(r)>0$ such that
  \begin{align}
    \left|x_{k+1} - (r-1) d \lambda^2 x_k\right| &\le C (r-1) d \lambda^2 x_k^2,\\
    \left|y_{k+1} - (r-1) d \lambda^4 y_k\right| &\le C x_k^2.
  \end{align}
  Take $k_0 = \Theta_r(1/\delta)$ such that $\frac 12 \le (1-2 C \delta)^{k_0}$ and $(1+2 C \delta)^{k_0} \le 2$. Then for all $0\le k\le k_0$, we have
  \begin{align}
    \frac 12 ((r-1) d \lambda^2)^{k} \delta \le x_{k} \le 2 ((r-1) d \lambda^2)^{k} \delta.
  \end{align}
  On the other hand,
  \begin{align}
    y_{k_0} &\le \sum_{0\le k\le k_0-1} C ((r-1) d \lambda^4)^{k_0-k-1} x_k^2 + ((r-1) d \lambda^4)^{k_0} y_0 \\
    \nonumber &\le 4 C \sum_{0\le k\le k_0-1} \lambda^{2(k_0-k-1)} ((r-1) d \lambda^2)^{k_0+k-1} \delta^2 + \lambda^{2 k_0} ((r-1) d \lambda^2)^{k_0} \delta.
  \end{align}
  Because $\lambda\in \left[-\frac 1{2^{r-1}-1}, \frac 15\right]$ and $(r-1) d \lambda^2 \ge 0.99$, we have $|\lambda| \le (r-1) d \lambda^2$.
  So
  \begin{align}
    \lambda^{2(k_0-k-1)} ((r-1) d \lambda^2)^{k_0+k-1} &\le |\lambda|^{k_0-k-1} ((r-1) d \lambda^2)^{2k_0-2} \\
    \nonumber &\le 1.1 |\lambda|^{k_0-k-1} ((r-1) d \lambda^2)^{2k_0}.
  \end{align}
  Furthermore, because $k_0=\Theta(1/\delta)=\omega(\log(1/\delta))$, we have $|\lambda|^{k_0} \le \delta$ (for small enough $\delta$).
  So
  \begin{align}
    \lambda^{2 k_0} ((r-1) d \lambda^2)^{k_0} \delta
    \le ((r-1) d \lambda^2)^{2 k_0} \delta^2.
  \end{align}
  Combining the above, we get
  \begin{align}
    y_{k_0} &\le \left( 4.4 C \sum_{0\le k\le k_0-1} |\lambda|^{k_0-k-1} + 1\right) \left(((r-1) d \lambda^2)^{k_0} \delta\right)^2 \\
    \nonumber & \le (9C+1) \left(((r-1) d \lambda^2)^{k_0} \delta\right)^2.
  \end{align}
  This finishes the proof.
\end{proof}

\subsection{Fine estimates} \label{sec:boht-app:fine}
The goal of this section is to prove Lemma~\ref{lem:boht-robust:fine}.

\begin{proof}[Proof of Lemma~\ref{lem:boht-robust:fine}, Eq.~\eqref{eqn:lem-boht-robust-fine:i}]
  Use the same notation as in the proof of Lemma~\ref{lem:boht-robust:a-priori}, Eq.~\eqref{eqn:lem-boht-robust-a-priori:i}.
  Expanding Eq.~\eqref{eqn:boht-app:a-priori:i-chi2-cap-R} using
  \begin{align} \label{eqn:boht-app:fine:i-expand}
    \frac 1{1+x}=1-x+x^2-x^3+\frac{x^4}{1+x},
  \end{align}
  we get
  \begin{align}
    C_{\chi^2}(R) = \frac{\lambda^2}{2^r}\left(T_0(\theta) - \frac{\lambda}{2} T_1(\theta) + \frac{\lambda^2}{4} T_2(\theta) - \frac{\lambda^3} 8 T_3(\theta) + \frac{\lambda^4}{16} R_4(\theta) \right),
  \end{align}
  where $T_0(\theta)$ is defined in Eq.~\eqref{eqn:boht-app:a-priori:i-T0}, $T_1(\theta)$ is defined in Eq.~\eqref{eqn:boht-app:a-priori:i-T1}, and the remaining are defined as
  \begin{align}
    T_2(\theta) &= \sum_{y\in \{\pm\}^{r-1}} \alpha(\theta,y)(\alpha(\theta,y)-\alpha(\theta,-y))(\alpha(\theta,y)+\alpha(\theta,-y)-2)^2,\\
    T_3(\theta) &= \sum_{y\in \{\pm\}^{r-1}} \alpha(\theta,y)(\alpha(\theta,y)-\alpha(\theta,-y))(\alpha(\theta,y)+\alpha(\theta,-y)-2)^3,\\
    R_4(\theta) &= \sum_{y\in \{\pm\}^{r-1}} \frac{\alpha(\theta,y)(\alpha(\theta,y)-\alpha(\theta,-y))(\alpha(\theta,y)+\alpha(\theta,-y)-2)^4}{1+\frac \lambda 2 (\alpha(\theta,y)+\alpha(\theta,-y)-2)}.
  \end{align}
  Note that
  \begin{align}
    \frac{\lambda}{2} |R_4(\theta)| &\le \sum_{y\in \{\pm\}^{r-1}} \alpha(\theta,y)(\alpha(\theta,y)+\alpha(\theta,-y)-2)^4 \left|\frac{R_+(y)-R_-(y)}{R_+(y)+R_-(y)}\right| \\
    \nonumber &\le \sum_{y\in \{\pm\}^{r-1}} \alpha(\theta,y)(\alpha(\theta,y)+\alpha(\theta,-y)-2)^4.
  \end{align}
  By Lemma~\ref{lem:boht-app:a-priori:moment}, we have
  \begin{align}
    \bE T_2(\theta) &= 2^{r-1} \left( (1+6x_k+y_k)^{r-1} - 4(1+3x_k)^{r-1}-(1-2x_k+y_k)^{r-1}+4(1+x_k)^{r-1}\right),\\
    \bE T_3(\theta) &= 2^{r-1} \left(
      (1+10x_k+5y_k)^{r-1} + (1+2x_k-3y_k)^{r-1} -6(1+6x_k+y_k)^{r-1} \right.\\
    \nonumber & \left.+12(1+3x_k)^{r-1} +4(1-2x_k+y_k)^{r-1}-8(1+x_k)^{r-1}-4(1-x_k)^{r-1}
    \right),\\
    \frac{\lambda}2 \bE |R_4(\theta)| &\le 2^{r-1} \left(
      (1+10x_k+5y_k)^{r-1}+5(1+2x_k-3y_k)^{r-1}-8(1+6x_k+y_k)^{r-1}\right.\\
    \nonumber & +10(1-2x_k+y_k)^{r-1}-32(1-y_k)^{r-1}+24(1+3x_k)^{r-1}-24(1-2x_k+y_k)^{r-1}\\
    \nonumber &\left.+40(1-x_k)^{r-1}-32(1+x_k)^{r-1}+16\right).
  \end{align}
  Expanding with respect to $x_k$ and using $y_k\le \gamma x_k^2$, we get
  \begin{align}
    \bE T_0(\theta) &= 2^r (r-1) x_k + O_r(x_k^3),\\
    \bE T_1(\theta) &= 2^{r+1} (r-1)(r-2) x_k^2 + O_r(x_k^3),\\
    \bE T_2(\theta) &= O_r(x_k^3),\\
    \bE T_3(\theta) &= O_r(x_k^3),\\
    \frac \lambda 2\bE |R_4(\theta)| &= O_r(x_k^3).
  \end{align}

  Finally,
  \begin{align}
    x'_{k+1} &= \bE_{\theta_1,\ldots,\theta_{r-1}} C_{\chi^2}((B_{\theta_1}\times \cdots \times B_{\theta_{r-1}})\circ B_{r,\lambda}) \\
    \nonumber &= \bE_{\theta_1,\ldots,\theta_{r-1}} \frac{\lambda^2}{2^r}\left(T_0(\theta) - \frac{\lambda}{2} T_1(\theta) + \frac{\lambda^2}{4} T_2(\theta) - \frac{\lambda^3} 8 T_3(\theta) + \frac{\lambda^4}{16} R_4(\theta) \right), \\
    \nonumber &= (r-1) \lambda^2 x_k - (r-1)(r-2) \lambda^3 x_k^2 + O_r(\lambda^2 x_k^3).
  \end{align}
  This finishes the proof.
\end{proof}

\begin{proof}[Proof of Lemma~\ref{lem:boht-robust:fine}, Eq.~\eqref{eqn:lem-boht-robust-fine:ii}]
  Use the same notation as in the proof of Lemma~\ref{lem:boht-robust:a-priori}, Eq.~\eqref{eqn:lem-boht-robust-a-priori:iii}.
  Expanding Eq.~\eqref{eqn:boht-app:a-priori:iii-chi2-cap-R} using Eq.~\eqref{eqn:boht-app:fine:i-expand},
  we get
  \begin{align}
    C_{\chi^2}(R) = \frac{1}{2^{b+1}}\left(T_0(\theta) - \frac{1}{2} T_1(\theta) + \frac{1}{4} T_2(\theta) - \frac{1} 8 T_3(\theta) + \frac{1}{16} R_4(\theta) \right),
  \end{align}
  where $T_0(\theta)$ is defined in Eq.~\eqref{eqn:boht-app:a-priori:iii-T0}, $T_1(\theta)$ is defined in Eq.~\eqref{eqn:boht-app:a-priori:iii-T1}, and the remaining are defined as
  \begin{align}
    T_2(\theta) &= \sum_{y\in \{\pm\}^b} \alpha(\theta,y)(\alpha(\theta,y)-\alpha(\theta,-y))(\alpha(\theta,y)+\alpha(\theta,-y)-2)^2,\\
    T_3(\theta) &= \sum_{y\in \{\pm\}^b} \alpha(\theta,y)(\alpha(\theta,y)-\alpha(\theta,-y))(\alpha(\theta,y)+\alpha(\theta,-y)-2)^3,\\
    R_4(\theta) &= \sum_{y\in \{\pm\}^b} \frac{\alpha(\theta,y)(\alpha(\theta,y)-\alpha(\theta,-y))(\alpha(\theta,y)+\alpha(\theta,-y)-2)^4}{1+\frac 1 2 (\alpha(\theta,y)+\alpha(\theta,-y)-2)}.
  \end{align}
  Using the same computation as the proof of Lemma~\ref{lem:boht-robust:fine}, Eq.~\eqref{eqn:lem-boht-robust-fine:i}, we get
  \begin{align}
    \bE_\theta T_2(\theta) &= 2^b \left( (1+6x'_{k+1}+y'_{k+1})^b - 4(1+3x'_{k+1})^b-(1-2x'_{k+1}+y'_{k+1})^b\right.\\
    &\left.+4(1+x'_{k+1})^b\right),\\
    \bE_\theta T_3(\theta) &= 2^b \left(
      (1+10x'_{k+1}+5y'_{k+1})^b + (1+2x'_{k+1}-3y'_{k+1})^b -6(1+6x'_{k+1}+y'_{k+1})^b \right.\\
    \nonumber & \left.+12(1+3x'_{k+1})^b +4(1-2x'_{k+1}+y'_{k+1})^b-8(1+x'_{k+1})^b-4(1-x'_{k+1})^b
    \right),\\
    \frac{1}2 \bE_\theta |R_4(\theta)| &\le 2^b \left(
      (1+10x'_{k+1}+5y'_{k+1})^b+5(1+2x'_{k+1}-3y'_{k+1})^b-8(1+6x'_{k+1}+y'_{k+1})^b\right.\\
    \nonumber & +10(1-2x'_{k+1}+y'_{k+1})^b-32(1-y'_{k+1})^b+24(1+3x'_{k+1})^b\\
    \nonumber &\left.-24(1-2x'_{k+1}+y'_{k+1})^b+40(1-x'_{k+1})^b-32(1+x'_{k+1})^b+16\right).
  \end{align}
  Because $y_k \le \gamma x_k^2$, using Eq.~\eqref{eqn:lem-boht-robust-a-priori:i} and Eq.~\eqref{eqn:lem-boht-robust-a-priori:ii} we see that $y'_{k+1} \le \gamma x_{k+1}^{\prime 2}$.
  Using Lemma~\ref{lem:boht-robust:low-deg-poly}, we get
  \begin{align}
    \bE_b \frac 1{2^{b+1}} \bE_\theta T_0(\theta) &= d x'_{k+1} + O_r(d^3 x_{k+1}^{\prime 3}),\\
    \bE_b \frac 1{2^{b+1}} \bE_\theta T_1(\theta) &= 2\left(\bE_{b\sim D} b(b-1)\right) x_{k+1}^{\prime 2} + O_r(d^3 x_{k+1}^{\prime 3}),\\
    \bE_b \frac 1{2^{b+1}} \bE_\theta T_2(\theta) &= O_r(d^3 x_{k+1}^{\prime 3}),\\
    \bE_b \frac 1{2^{b+1}} \bE_\theta T_3(\theta) &= O_r(d^3 x_{k+1}^{\prime 3}),\\
    \bE_b \frac 1{2^{b+1}} \bE_\theta |R_4(\theta)| &= O_r(d^3 x_{k+1}^{\prime 3}).
  \end{align}

  Finally,
  \begin{align}
    x_{k+1} &= \bE_{\theta_1,\ldots,\theta_b} C_{\chi^2}(B_{\theta_1}\star \cdots \star B_{\theta_b}) \\
    \nonumber &= \bE_{\theta_1,\ldots,\theta_b} \frac{1}{2^{b+1}}\left(T_0(\theta) - \frac{1}{2} T_1(\theta) + \frac{1}{4} T_2(\theta) - \frac{1} 8 T_3(\theta) + \frac{1}{16} R_4(\theta) \right), \\
    \nonumber &= d x'_{k+1} - \left(\bE_{b\sim D} b(b-1)\right) x_{k+1}^{\prime 2} + O_r(d^3 x_{k+1}^{\prime 3}).
  \end{align}
  This finishes the proof.
\end{proof}

\begin{proof}[Proof of Lemma~\ref{lem:boht-robust:fine}]
  Eq.~\eqref{eqn:lem-boht-robust-fine:final} follows directly from Eq.~\eqref{eqn:lem-boht-robust-fine:i} and Eq.~\eqref{eqn:lem-boht-robust-fine:ii}.
  It remains to prove that $M_{P,k+1}$ is $\gamma$-normal.

  Recall our assumption that $\lambda\in \left[-\frac 1{2^{r-1}-1}, \frac 15\right]$ and $(r-1) d \lambda^2 \ge 0.99$.
  By Eq.~\eqref{eqn:lem-boht-robust-a-priori:final-i} and Eq.~\eqref{eqn:lem-boht-robust-a-priori:final-ii}, there exists $C=C(r)>0$ such that
  \begin{align}
    |x_{k+1}-(r-1) d \lambda^2 x_k| &\le C(r-1) d \lambda^2 x_k^2,\\
    |y_{k+1}-(r-1) d \lambda^4 y_k| &\le C x_k^2.
  \end{align}
  If $x_k$ is small enough, then
  \begin{align}
    x_{k+1} \ge ((r-1) d \lambda^2 - 0.01) x_k \ge 0.98 x_k.
  \end{align}
  On the other hand,
  \begin{align}
    y_{k+1} &\le (r-1) d \lambda^4 y_k + C x_k^2 \\
    \nonumber &\le (r-1) d \lambda^4 \gamma x_k^2 + C x_k^2 \\
    \nonumber &\le \lambda^2 \gamma x_k^2 + C x_k^2 \\
    \nonumber &\le 0.5 \gamma x_k^2 + C x_k^2,
  \end{align}
  where the last step uses the assumption that $|\lambda| \le \max\left\{\frac 1{2^{r-1}-1}, \frac 15\right\}$.
  Note that $C$ does not depend on the choice of $\gamma$.
  Therefore we can take $\gamma$ large enough such that $0.5\gamma + C \le \gamma \cdot 0.98^2$.
  In this way, $y_{k+1} \le \gamma x_{k+1}^2$.
\end{proof}

\section{Proofs of remaining results} \label{sec:remain}
In this section we prove our remaining main results, Corollary~\ref{coro:hsbm-upper} (non-tightness of the KS threshold), Theorem~\ref{thm:boht-asymp} (asymptotic BOHT reconstruction threshold) and Theorem~\ref{thm:hsbm-asymp} (asymptotic HSBM weak recovery bounds).

\subsection{Non-tightness of the KS threshold} \label{sec:remain:hsbm-upper}
\begin{proof}[Proof of Corollary~\ref{coro:hsbm-upper}]
  For $\lambda\in \left[-\frac 1{2^{r-1}-1}, 1\right]$, define
  \begin{align} \label{eqn:proof-coro-hsbm-upper:f}
    f_r(\lambda) = d_{\KL}\left( \lambda + \frac{1-\lambda}{2^{r-1}} \left\| \frac 1{2^{r-1}}\right. \right).
  \end{align}
  By Theorem~\ref{thm:hsbm-upper}, if for some $(r,\lambda)$ we have
  \begin{align} \label{eqn:proof-coro-hsbm-upper:cond}
    f_r(\lambda) > r(r-1)\lambda^2 \log 2,
  \end{align}
  then for this $(r,\lambda)$ the KS threshold is not tight for HSBM weak recovery.

  For $r=5,6$, we take $\lambda_0$ as in Table~\ref{tab:hsbm-upper} and verify that Eq.~\eqref{eqn:proof-coro-hsbm-upper:cond} holds for all $\lambda\in \left[-\frac 1{2^{r-1}-1}, \lambda_0\right)$. Therefore the KS threshold is not tight for such $(r,\lambda)$.

  It remains to prove that for all $r\ge 7$, there exists $\lambda_0(r)>0$ such that Eq.~\eqref{eqn:proof-coro-hsbm-upper:cond} holds for all $\lambda\in \left[-\frac 1{2^{r-1}-1}, \lambda_0\right)$.
  For $\lambda\in \left[-\frac 1{2^{r-1}-1},0\right)$, we apply Lemma~\ref{lem:d-kl-technical} with $x=\lambda + \frac{1-\lambda}{2^{r-1}}$, $y=\frac 1{2^{r-1}}$ and get
  \begin{align}
    f_r(\lambda) \ge \frac{2^{r-1}-1}{2} \lambda^2 > r(r-1) \lambda^2 \log 2
  \end{align}
  where the last step holds for all $r\ge 7$.

  For $\lambda>0$, we observe that
  \begin{align} \label{eqn:proof-coro-hsbm-upper:f-lim}
    \lim_{\lambda\to 0} \frac 1{\lambda^2} f_r(\lambda) = f''_r(0) = \frac{2^{r-1}-1}{2}.
  \end{align}
  So for $r\ge 7$, there exists $\lambda_0=\lambda_0(r)>0$ such that Eq.~\eqref{eqn:proof-coro-hsbm-upper:cond} holds for all $\lambda\in \left(0, \lambda_0\right)$.
  Combining both cases we finish the proof.
\end{proof}

\begin{lemma} \label{lem:d-kl-technical}
  For $0\le x\le y\le \frac 12$, we have
  \begin{align}
    d_{\KL}\left(x\| y\right) \ge \frac{(x-y)^2}{2y(1-y)}.
  \end{align}
\end{lemma}
\begin{proof}
  We first consider the case $x>0$.
  Let $t=y-x$. Then
  \begin{align}
    d_{\KL}\left(x\| y\right) &= x \log \left( 1 - \frac{t}{y}\right) + (1-x) \log \left( 1 + \frac {t}{1-y}\right) \\
    \nonumber &= (-y+t) \sum_{k\ge 1} \frac 1k \left(\frac{t}y\right)^k
    + (1-y+t) \sum_{k\ge 1} (-1)^{k-1}\frac 1k \left(\frac{t}{1-y}\right)^k \\
    \nonumber &= \left(-t+\sum_{k\ge 2} \frac {t^k y^{-k+1}}{k(k-1)}\right)  + \left( t + \sum_{k\ge 2} (-1)^k \frac{t^k (1-y)^{-k+1}}{k(k-1)} \right)\\
    \nonumber &= \sum_{k\ge 2} \frac {t^k}{k(k-1)} \left( y^{-k+1} + (-1)^k (1-y)^{-k+1} \right).
  \end{align}
  Because $y\le \frac 12$, the summand is non-negative for all $k\ge 2$.
  Taking $k=2$, we get
  \begin{align}
    d_{\KL}\left(x\| y\right) \ge \frac 12 t^2 \left( y^{-1} + (1-y)^{-1}\right) = \frac {(y-x)^2}{2y(1-y)}.
  \end{align}

  The case $x=0$ follows by continuity.
\end{proof}

\subsection{Asymptotic BOHT reconstruction threshold} \label{sec:remain:boht-asymp}
\begin{proof}[Proof of Theorem~\ref{thm:boht-asymp}]
  \textbf{Upper bound.}
  Let $w=w^*+\epsilon$.
  By Lemma~\ref{lem:boht-monotone}, it suffices to prove for $(r-1) d \lambda^2 = w$.
  By Lemma~\ref{lem:large-step-g-property}\ref{item:lem-large-step-g-property:i} and definition of $w^*$, there exists $x_0\in (0,1]$ such that $g_{r,w}(x_0) > x_0$.
  Let $\epsilon = g_{r,w}(x_0)-x_0$.
  Take $d_0=d_0(r,\epsilon)$ in Prop.~\ref{prop:large-deg-asymp}.
  By Prop.~\ref{prop:large-deg-asymp}, for $d\ge d_0$ and $(r-1) d \lambda^2 = w$, for any BMS channel $P$ with $C_{\chi^2}(P) \ge x_0$, we have
  \begin{align}
    C_{\chi^2}(\BP(P)) \ge g_{r,w}(C_{\chi^2}(P))-\epsilon \ge g_{r,w}(x_0)-\epsilon = x_0,
  \end{align}
  where the second step is by Lemma~\ref{lem:large-step-g-property}\ref{item:lem-large-step-g-property:i}.
  Therefore $\lim_{k\to \infty} C_{\chi^2}(M_k) \ge x_0>0$ and reconstruction is possible.

  \textbf{Lower bound.}
  Let $w=w^*-\epsilon$.
  By Lemma~\ref{lem:large-step-g-property}\ref{item:lem-large-step-g-property:i} and definition of $w^*$, we have $g_{r,w}(x)<x$ for all $x\in (0,1]$.
  By the same proof as Prop.~\ref{prop:large-deg-asymp-r56}, for any $\epsilon'>0$ there exists $d_0=d_0(r,\epsilon')$ such that $\lim_{k\to \infty} C_{\chi^2}(M_k) < \epsilon'$.
  The rest of the proof is the same as the proof of Theorem~\ref{thm:boht-r56}.
\end{proof}

\subsection{Asymptotic HSBM weak recovery bounds} \label{sec:remain:hsbm-asymp}
\begin{proof}[Proof of Theorem~\ref{thm:hsbm-asymp}\ref{item:thm-hsbm-asymp:lower}]
  We prove that $w^*(r) \ge \frac 1{2^{r-2}}$, where $w^*(r)$ is defined in Eq.~\eqref{eqn:thm-boht-asymp-wstar}. That is, for $w=\frac 1{2^{r-2}}$, $g_{r,w}(x)\le x$ for all $0\le x\le 1$.
  By Lemma~\ref{lem:large-step-g-property}\ref{item:lem-large-step-g-property:ii}, $g_{r,w}(x) \le s_{r,w}(x)$.
  Therefore it suffices to prove that $s_{r,w}(x)\le x$.
  Note That
  \begin{align}
    \frac{s_{r,w}(x)}x = w \cdot \sum_{0\le i\le \left\lfloor\frac{r-1}2\right\rfloor} \binom{r-1}{2i+1} x^{2i}
  \end{align}
  is a polynomial with non-negative coefficients.
  Therefore for $0\le x\le 1$, $\frac{s_{r,w}(x)}x$ is maximized at $x=1$, taking value $w 2^{r-2} = 1$.
  So $w^*(r) \ge w$.
  Applying Theorem~\ref{thm:boht-asymp} and {\cite[Theorem 5.15]{gu2023channel}} we finish the proof.
\end{proof}

\begin{proof}[Proof of Theorem~\ref{thm:hsbm-asymp}\ref{item:thm-hsbm-asymp:upper}]
  Recall function $f_r$ defined in Eq.~\eqref{eqn:proof-coro-hsbm-upper:f}.
  By Eq.~\eqref{eqn:proof-coro-hsbm-upper:f-lim}, for any $\delta>0$, there exists $\lambda_0=\lambda_0(r,\delta)>0$ such that for all $|\lambda|\le \lambda_0$, we have $f_r(\lambda) > \left(\frac{2^{r-1}-1}2-\delta\right) \lambda^2$.
  Furthermore, because $f_r(\lambda)$ is monotone increasing for $\lambda>0$ and monotone decreasing for $\lambda<0$, we can take $\lambda_0$ such that $f_r(\lambda) > \left(\frac{2^{r-1}-1}2-\delta\right) \min\{\lambda^2,\lambda_0^2\}$ for all $\lambda\in \left[-\frac 1{2^{r-1}-1}, 1\right]$.

  Now fix $\epsilon>0$.
  Take $\delta>0$ such that $\left(\frac{2r\log 2}{2^{r-1}-1} + \epsilon\right) \cdot r^{-1}\left(\frac{2^{r-1}-1}{2}-\delta\right) = \log 2$.
  Take $\lambda_0 = \lambda_0(r,\delta)$.
  Take $d_0 = \lambda_0^{-2} \left(\frac{2r\log 2}{2^{r-1}-1} + \epsilon\right)$.
  Then for any $d\ge d_0$ and $d \lambda^2 \ge \frac{2r\log 2}{2^{r-1}-1} + \epsilon$, we have
  \begin{align}
    \frac dr \cdot f_r(\lambda) > \frac dr \cdot \left(\frac{2^{r-1}-1}2-\delta\right) \min\{\lambda^2,\lambda_0^2\}
    = \frac{d \min\{\lambda^2, \lambda_0^2\}\log 2}{\frac{2r\log 2}{2^{r-1}-1} + \epsilon} \ge \log 2.
  \end{align}
  By Theorem~\ref{thm:hsbm-upper}, we finish the proof.
\end{proof}

\section{Discussions} \label{sec:discuss}


\subsection{Comparison with the \texorpdfstring{$q$}{q}-community SBM} \label{sec:discuss:q-sbm}
The HSBM is a generalization of the two-community SBM to hypergraphs. Another natural generalization is the $q$-community SBM ($q$-SBM). In this model, $n$ vertices are given i.i.d.~$\Unif([q])$ labels $X_v$. An edge $(u,v)$ is added with probability $\frac an$ if $X_u=X_v$, and with probability $\frac bn$ if $X_u\ne X_v$. The goal of the weak recovery problem is to recover a non-trivial fraction of the labels given the generated graph.

For the $q$-SBM, define
\begin{align}
  d = \frac{a+(q-1)b}q, \qquad \lambda=\frac{a-b}{a+(q-1)b}.
\end{align}
The Kesten-Stigum threshold is at $d \lambda^2=1$. It plays a similar role in the $q$-SBM as in the HSBM: weak recovery is always possible above the KS threshold, and this is sometimes, but not always, tight.

\begin{table}[ht]
  \centering
  \begin{tabular}{|c|c|c|c|} \hline
    & any $d$ &  small $d$ & large $d$\\ \hline
    $q=3$, $\lambda>0$ & unknown & unknown & tight \\
    $q=3$, $\lambda<0$ & unknown & unknown & tight \\
    $q=4$, $\lambda>0$ & unknown & unknown & tight \\
    $q=4$, $\lambda<0$ & unknown & not tight & tight \\
    $5\le q\le 10$, $\lambda>0$ & unknown & unknown & unknown \\
    $5\le q\le 10$, $\lambda<0$ & unknown & not tight & unknown \\
    $q\ge 11$, $\lambda>0$ & unknown & unknown & not tight \\
    $q\ge 11$, $\lambda<0$ & never tight & (not tight) & (not tight) \\ \hline
  \end{tabular}
  \caption{Tightness of the KS threshold for symmetric $q$-SBM weak recovery. An entry in parentheses means it is implied by another entry in the same row. References: Possibility above KS: \cite{abbe2015detection}. Tightness results: \cite{mossel2023exact}. Non-tightness results: \cite{abbe2015detection} (all non-tightness results) and \cite{banks2016information} (for $q\ge 5$).}
  \label{tab:tightness-ks-q-sbm}
\end{table}

Our current knowledge about tightness of the KS threshold for $q$-SBM weak recovery is summarized in Table~\ref{tab:tightness-ks-q-sbm}.
Comparing Table~\ref{tab:tightness-ks} and Table~\ref{tab:tightness-ks-q-sbm}, we make a few interesting observations.
In both models, in the large $d$ regime, as $r$ and $q$ goes larger, the KS threshold goes from tight to non-tight. However, the HSBM has a small $d$ regime ($\lambda\ge \frac 15$) where KS is tight for any $r$, but the $q$-SBM is not known to have such a regime.

For the $4$-SBM, it is conjectured that tightness of the KS threshold is different for $\lambda>0$ and $\lambda<0$ cases. It is known that for $\lambda<0$ and small $d$, KS is not tight, and conjectured that for $\lambda>0$, KS is always tight. We expect similar behavior to hold for the HSBM with $r=5,6$.

Unlike the HSBM case, we know more about tightness of the KS threshold for $q$-ary Potts model on a tree than the $q$-SBM. For $q\ge 5$, it is known that the KS threshold is never tight (\cite{sly2011reconstruction,mossel2023exact}) for the Potts model. However, as shown in Table~\ref{tab:tightness-ks-q-sbm}, there are several regimes with $q\ge 5$ where the KS threshold is not known to be not tight for the $q$-SBM.

\subsection{Methods for reconstruction on trees and hypertrees} \label{sec:discuss:non-recon}
As noted in \cite{mossel2004survey}, there are a number of different methods for establishing the reconstruction threshold for the Ising model, but few, if any, generalize to other BOT problems. After that, new methods are developed for proving bounds on the reconstruction threshold.

For the upper bound (reconstruction), the Kesten-Stigum bound is simple and works universally. When it is not tight, upper bounds are proved by establishing non-contraction properties.
Suppose for a certain non-trivial channel $P$, we have $\BP(P)\ge P$, where $\ge$ is a certain channel preorder that respects the $\BP$ operator (which is often the degradation preorder), then by iterating we get $\BP^k(\Id) \ge P$ for all $k\ge 0$, and reconstruction is possible.
\cite{sly2009reconstruction} used freezing probability as the information measure to prove reconstruction results for the random coloring model, which can be viewed as taking $\ge$ to be the degradation preorder and $P$ to be an erasure channel.
\cite{sly2011reconstruction,mossel2023exact} used $\chi^2$-capacity (also called magnetization) to prove reconstruction results for the Potts model in the large degree regime (via the large degree asymptotics method). This was also used in \cite{liu2019large} for the asymmetric Ising model and in \cite{gu2023weak} for the BOHT model. The current work uses this method in Theorem~\ref{thm:boht-asymp}.
Sometimes the channel $P$ is designed explicitly, such as \cite{sly2016reconstruction,zhang2017phase} which established lower order terms of the reconstruction threshold for the random coloring model and the NAE-$k$-SAT model (in other words, the BOHT model with $\lambda=-\frac 1{2^{r-1}-1}$).

For the lower bound (non-reconstruction), there are two particularly successful methods.
One method proves contraction of certain information measures using strong data processing inequalities (SDPIs) and subadditivity. \cite{borgs2006kesten} proved tightness of the KS threshold for roughly symmetric Ising models using $\chi^2$-information. \cite{kulske2009symmetric} proved non-reconstruction results for Potts model using this method with symmetric KL (SKL) information, and \cite{gu2023non} obtained improved results using mutual information.
\cite{gu2023weak} proved non-reconstruction results for the BOHT model using a multi-terminal generalization of the method with SKL information and $\chi^2$-information.
The advantage of this method is that it is often simple to obtain, works for any degree, and can be easily generalized to different tree structures (non-Galton-Watson trees). However, the choice of the information measure is mysterious and it is unclear what is the best information measure to use.

The other method is Sly's method, which uses large degree asymptotics and robust reconstruction. This method was introduced in \cite{sly2009reconstruction,sly2011reconstruction} for the random coloring model and the Potts model. \cite{liu2019large} used the method for the asymmetric Ising model. \cite{mossel2023exact} improved this method by analyzing more orders in the Taylor expansion, and removing the need for a certain concentration result. The current work uses a hypertree verion of Sly's method in Theorem~\ref{thm:boht-r56}. The advantage of this method is that it often gives tight results. The disadvantage is that it needs a large degree assumption and requires heavy computation.

\subsection{Further directions}
\paragraph{Tightness of the KS threshold}
The obvious open question is to fill in the blanks of Table~\ref{tab:tightness-ks} and Table~\ref{tab:tightness-ks-q-sbm}. It could be possible that using a clever information measure in the SDPI method gives some tight non-reconstruction results. However, it is unclear what could be such an information measure, other than mutual information, $\chi^2$-information, and SKL information.

For $r=5,6$ and $\lambda>0$, Table~\ref{tab:tightness-ks} shows that the KS threshold is tight for small $d$ and large $d$. If there is a certain kind of convexity property, then we could conclude that KS is tight for all $d$.

For $r\ge 7$ and $\lambda>0$, Table~\ref{tab:tightness-ks} shows that the KS threshold is tight for small $d$ and not tight for large $d$. Therefore, it seems possible that there exists $d^*=d^*(r)$ for which KS is tight for $d<d^*$ and not tight for $d>d^*$.

\paragraph{Asymptotic HSBM weak recovery bounds}
Theorem~\ref{thm:hsbm-asymp} gives asymptotic bounds on the weak recovery threshold, where there is a factor $\Theta(r)$ gap.
In \cite{banks2016information}, similar asymptotic bounds were obtained for the $q$-SBM, and the gap there is of a factor of two.
Reducing the gaps in the asymptotic bounds for HSBM is an interesting question.
Evidences suggest that the upper bound (Theorem~\ref{thm:hsbm-asymp}\ref{item:thm-hsbm-asymp:upper}) might be closer to the true threshold as $r\to \infty$. The lower bound (Theorem~\ref{thm:hsbm-asymp}\ref{item:thm-hsbm-asymp:upper}) could be not tight for two reasons. One is that the HSBM weak recovery threshold may be different from the BOHT reconstruction threhsold; the other is that the constant $\frac 1{2^{r-2}}$ is not an asymptotically tight bound of $w^*(r)$ in Eq.~\eqref{eqn:thm-boht-asymp-wstar}. We choose $\frac 1{2^{r-2}}$ so that the statement of Theorem~\ref{thm:hsbm-asymp}\ref{item:thm-hsbm-asymp:upper} works for any $r\ge 3$ rather than as $r\to \infty$.

\paragraph{Non-uniform HSBM}
An interesting direction of generalization is the non-uniform HSBMs, where hyperedges may have different sizes. This model has a sequence of non-negative real parameters $(a_r,b_r)_{r\ge 2}$ with finitely many non-zero entries (not depending on $n$, the number of vertices). First, the vertices are given i.i.d.~$\Unif(\{\pm\})$ labels.
Then each set $S$ of vertices with $|S|\ge 2$ is added as an hyperedge with probability $\frac{a_{|S|}}{\binom n{|S|-1}}$ or $\frac{b_{|S|}}{\binom n{|S|-1}}$ depending on whether the set $S$ is monochromatic. The KS threshold is defined as $\sum_{r\ge 2} (r-1) d_r \lambda_r^2 = 1$ where $d_r$ and $\lambda_r$ are defined from $a_r$ and $b_r$ as in Eq.~\eqref{eqn:hsbm-params-d-lambda}. \cite{chodrow2023nonbacktracking} conjectured that weak recovery is possible above the KS threshold, and partial progress was made in \cite{dumitriu2021partial}. By analyzing the corresponding non-uniform BOHT model, one could show that weak recovery is impossible below the KS threshold under certain conditions. For example, by adapting the multi-terminal SDPI method from \cite{gu2023weak}, one can show that if $a_r\ge b_r$ for $r\le 4$ and $a_r=b_r=0$ for $r\ge 5$, then weak recovery is impossible below the KS threshold. Furthermore, our robust reconstruction result (Theorem~\ref{thm:boht-robust}) seems to still hold in the non-uniform case.
On the other hand, by adapting our proof of Theorem~\ref{thm:hsbm-upper} one could prove that weak recovery is possible whenever
\begin{align}
  \sum_{r\ge 2} \frac {d_r} r \cdot d_{\KL}\left( \lambda_r + \frac{1-\lambda_r}{2^{r-1}} \left\| \frac 1{2^{r-1}} \right. \right) > \log 2.
\end{align}

\end{document}